 \documentclass[final, 1p, times]{elsarticle}
\usepackage{amssymb}
\usepackage{amsthm}
\usepackage{amsmath}
\usepackage{MnSymbol}
\usepackage{lineno}
\newtheorem{theorem}{Theorem}[section]
\newtheorem{lemma}[theorem]{Lemma}

\theoremstyle{definition}
\newtheorem{definition}[theorem]{Definition}

\newtheorem{proposition}[theorem]{Proposition}

\theoremstyle{remark}
\newtheorem{remark}[theorem]{Remark}

\theoremstyle{notation}

\theoremstyle{construction}
\newtheorem{construction}[theorem]{Construction}
\newtheorem{corollary}[theorem]{Corollary}
\numberwithin{equation}{section}
\journal{JSL}
\begin{document}

\begin{frontmatter}

%% Title,  authors and addresses


\title{The logic of pseudo-uninorms and their residua\tnoteref{mytitlenote}}
\tnotetext[mytitlenote]{This work is supported by the National Foundation of Natural Sciences of China (Grant No:  61379018 \&61662044\& 11571013)}
%% \author[label1, label2]{}

\author{SanMin Wang}

\address{Faculty of Science,  Zhejiang Sci-Tech University,  Hangzhou 310018,  P.R. China

Email: wangsanmin@hotmail.com}

\begin{abstract}
Our method  of density elimination is generalized to the non-commutative substructural logic ${\rm {\bf GpsUL}}^\ast$.  Then the standard
completeness of ${\rm {\bf HpsUL}}^\ast $ follows as a lemma by virtue of previous work by Metcalfe and Montagna.  This result shows that ${\rm {\bf HpsUL}}^\ast $
is the logic of pseudo-uninorms and their residua and answered the question posed by Prof. Metcalfe,  Olivetti,   Gabbay  and Tsinakis.
\end{abstract}


\end{frontmatter}
\setlength\linenumbersep{1cm}
\renewcommand\linenumberfont{\normalfont}
%\leftlinenumbers
%\pagewiselinenumbers

\section{Introduction}

In 2009,  Prof. Metcalfe,  Olivetti and Gabbay conjectured that the Hilbert
system ${\rm {\bf HpsUL}}$ is the logic of pseudo-uninorms and
their residua [1].  Although ${\rm {\bf HpsUL}}$ is the logic of bounded representable residuated lattices,   it is not the case,  as shown by Prof. Wang and Zhao in [2].  In 2013,  we constructed the system ${\rm {\bf
HpsUL}}^\ast $ by adding the weakly commutativity rule \[(WCM) \vdash
(A\rightsquigarrow t)\to (A\to t)\] to ${\rm {\bf HpsUL}}$ and conjectured
that it is the logic of residuated pseudo-uninorms and their
residua [3].

In this paper,  we prove our conjecture by showing that the density
elimination holds for the hypersequent system ${\rm {\bf GpsUL}}^{\rm {\bf
\ast }}$ corresponding to ${\rm {\bf HpsUL}}^\ast $. Then the
standard completeness of ${\rm {\bf HpsUL}}^\ast $ follows as a
lemma by virtue of previous work by Metcalfe and Montagna [4]. This shows
that ${\rm {\bf HpsUL}}^\ast $ is an axiomatization for the variety
of residuated lattices generated by all dense residuated chains.  Thus we
also answered the question posed by Prof. Metcalfe and Tsinakis in [5] in
2017.

In proving the density elimination for ${\rm {\bf GpsUL}}^{\rm {\bf \ast
}}$,  we have to overcome several difficulties as follows. Firstly,
cut-elimination doesn't holds for ${\rm {\bf GpsUL}}^\ast $. Note that
$(WCM)$ and the density rule$(D)$ are formulated as
\[\cfrac{G\vert \Gamma
, \Delta \Rightarrow t}{G\vert \Delta, \Gamma \Rightarrow t}\, \, \, , \quad\cfrac{G\vert \Pi \Rightarrow p\vert \Gamma, p, \Delta \Rightarrow B}{G\vert \Gamma, \Pi, \Delta \Rightarrow B}\]
in ${\rm {\bf GpsUL}}^\ast $,
respectively.  Consider the following derivation fragment.
\[
\boxed {\cfrac{\cfrac{ \ddots  \vdots  {\mathinner{\mkern0.5mu\raise0.7pt\hbox{.}\mkern0.0mu
 \raise2.1pt\hbox{.}\mkern0.8mu\raise3.6pt\hbox{.}}}}{G_1 \vert \Gamma _1
, t, \Delta _1 \Rightarrow A}\quad
\cfrac{\cfrac{ \ddots  \vdots  {\mathinner{\mkern0.5mu\raise0.7pt\hbox{.}\mkern0.0mu
 \raise2.1pt\hbox{.}\mkern0.8mu\raise3.6pt\hbox{.}}}}{G_2 \vert \Gamma _2
, \Delta _2 \Rightarrow t}}{G_2 \vert \Delta _2, \Gamma _2 \Rightarrow
t}(WCM)}{G_1 \vert G_2 \vert \Gamma _1, \Delta _2, \Gamma _2, \Delta _1
\Rightarrow A}(CUT)}.
\]
By the induction hypothesis of the proof of cut-elimination,  we get that\\
$G_1 \vert G_2 \vert \Gamma _1, \Gamma _2, \Delta _2, \Delta _1 \Rightarrow
A$ from $G_2 \vert \Gamma _2, \Delta _2 \Rightarrow t$ and $G_1 \vert \Gamma
_1, t, \Delta _1 \Rightarrow A$ by $(CUT)$. But we can't deduce $G_1 \vert
G_2 \vert \Gamma _1, \Delta _2, \Gamma _2, \Delta _1 \Rightarrow A$ from
$G_1 \vert G_2 \vert \Gamma _1, \Gamma _2, \Delta _2, \Delta _1 \Rightarrow
A$ by  $(WCM)$. We overcome this difficulty by introducing the following weakly cut rule into ${\rm {\bf GpsUL}}^\ast $

$$\cfrac{G_1 \vert \Gamma, t, \Delta \Rightarrow A \quad G_2
\vert \Pi \Rightarrow t}{G_1 \vert G_2 \vert \Gamma
, \Pi, \Delta \Rightarrow A}(WCT).$$

Secondly,  the proof of the density elimination for ${\rm {\bf GpsUL}}^{\rm
{\bf \ast }}$ become troublesome even for some simple cases in ${\rm {\bf
GUL}}$ [4]. Consider the following derivation fragment
\[
\boxed {\cfrac{\cfrac{\cfrac{ \ddots  \vdots  {\mathinner{\mkern0.5mu\raise0.7pt\hbox{.}\mkern0.0mu
 \raise2.1pt\hbox{.}\mkern0.8mu\raise3.6pt\hbox{.}}}}{G_1 \vert \Gamma _1
, \Pi _1, \Sigma _1 \Rightarrow A_1 }\quad\cfrac{ \ddots  \vdots  {\mathinner{\mkern0.5mu\raise0.7pt\hbox{.}\mkern0.0mu
 \raise2.1pt\hbox{.}\mkern0.8mu\raise3.6pt\hbox{.}}}}{G_2 \vert \Gamma _2
, \Pi _2 ', p, \Pi _2 '', \Sigma _2 \Rightarrow p}{\kern
1pt}{\kern
1pt}}{G_1 \vert G_2 \vert \Gamma _1, \Pi _2 ', p, \Pi _2
'', \Sigma _1 \Rightarrow A_1 \vert \Gamma _2, {\kern
1pt}\Pi _1, \Sigma _2 \Rightarrow p{\kern
1pt}}(COM)}{G_1 \vert G_2 \vert
\Gamma _1, \Pi _2 ', \Gamma _2, \Pi _1 {\kern
1pt}, \Sigma _2, \Pi _2 '', \Sigma _1 \Rightarrow A_1 }(D)}.
\]
Here,  the major problem is how to extend $(D)$ such that it is applicable to
$G_2 \vert \Gamma _2, \Pi _2 ', p, \Pi _2 '', \Sigma _2 \Rightarrow p$. By
replacing $p$ with $t$,  we get $G_2 \vert \Gamma _2, \Pi _2 ', t, \Pi _2
'', \Sigma _2 \Rightarrow t$. But there exists no derivation of $G_1 \vert
G_2 \vert \Gamma _1, \Pi _2 ', \Gamma _2, \Pi _1 {\kern
1pt}, \Sigma _2, \Pi _2 '', \Sigma _1 \Rightarrow A_1 $ from $G_2 \vert \Gamma
_2, \Pi _2 ', \Pi _2 '', \Sigma _2 \Rightarrow t$ and $G_1 \vert \Gamma _1
, \Pi _1, \Sigma _1 \Rightarrow A_1 $. Notice that $\Gamma _2, \Pi _2 '$ and
$\Pi _2 '', \Sigma _2 $ in $G_2 \vert \Gamma _2, \Pi _2 ', p, \Pi _2 '', \Sigma
_2 \Rightarrow p$ are commutated simultaneously in\\
 $G_1 \vert G_2 \vert
\Gamma _1, \Pi _2 ', \Gamma _2, \Pi _1, \Sigma _2, \Pi _2 '', \Sigma _1 \Rightarrow A_1 $,  which we can't obtain
by  $(WCM)$. It seems that $(WCM)$ can't be
strengthened further in order to solve this difficulty. We overcome this
difficulty by introducing a restricted subsystem ${\rm {\bf GpsUL}}_\Omega $
of ${\rm {\bf GpsUL}}^{\rm {\bf \ast }}$. ${\rm {\bf GpsUL}}_\Omega $ is a
generalization of ${\rm {\bf GIUL}}_\Omega $,  which we introduced in [6] in
order to solve a longstanding open problem,  i.e.,  the standard completeness
of ${\rm {\bf IUL}}$. Two new manipulations,  which we call the
derivation-splitting operation and derivation-splicing operation,  are
introduced in ${\rm {\bf GpsUL}}_\Omega $.

The third difficulty we encounter is that the conditions of applying the
restricted external contraction rule $(EC_\Omega )$ become more complex in
${\rm {\bf GpsUL}}_\Omega $ because new derivation-splitting operations make
the conclusion of the generalized density rule to be a set of hypersequents
rather than one hypersequent. We continue to apply derivation-grafting
operations in the separation algorithm of the multiple branches of ${\rm
{\bf GIUL}}_\Omega $ in [6] but we have to introduce a new construction method for ${\rm {\bf GpsUL}}_\Omega $ by induction on the height of the complete set
of maximal $(pEC)$-nodes other than on the number of branches.

\section{${\rm {\bf GpsUL}}$,  ${\rm {\bf GpsUL}}^\ast $ and ${\rm {\bf
GpsUL}}_{\Omega } $ }

\begin{definition} ([1])
${\rm {\bf GpsUL}}$ consist of the following initial sequents and rules:

\noindent\textbf{Initial sequents}
\[
\cfrac{}{A\Rightarrow A}(ID)
\quad
\cfrac{}{\Rightarrow t}(t_r )
\quad\cfrac{}{\Gamma, \bot, \Delta \Rightarrow A} (\bot _l )
\quad
\cfrac{}{\Gamma \Rightarrow \top}(\top _r )
\]
\noindent\textbf{Structural Rules}
\[
\cfrac{G\vert \Gamma \Rightarrow A\vert \Gamma \Rightarrow A}{G\vert \Gamma
\Rightarrow A}(EC)
\quad
\cfrac{G}{G\vert \Gamma \Rightarrow A}(EW)
\]
\[
\cfrac{G_1 \vert \Gamma _1, \Pi _1, \Delta _1 \Rightarrow A_1
\quad G_2 \vert \Gamma _2, \Pi _2, \Delta _2 \Rightarrow
A_2 }{G_1 \vert G_2 \vert \Gamma _1, \Pi _2, \Delta _1 \Rightarrow A_1
\vert \Gamma _2, {\kern 1pt}\Pi _1, \Delta _2 \Rightarrow A_2 }(COM)
\]
\noindent\textbf{Logical Rules}\\
\begin{minipage}[l]{0.49\linewidth}
\begin{align*}
\cfrac{G_1 \vert \Gamma \Rightarrow A\quad G_2 \vert \Delta \Rightarrow
B}{G_1 \vert G_2 \vert \Gamma, \Delta \Rightarrow A\odot B}(\odot _r )&\\
\cfrac{G_1 \vert \Gamma, B, \Delta \Rightarrow C\quad G_2 \vert \Pi
\Rightarrow A}{G_1 \vert G_2 \vert \Gamma, \Pi, A\to B, \Delta
\Rightarrow C}(\to _l )&\\
\cfrac{G_1 \vert \Pi \Rightarrow A\quad G_2 \vert \Gamma, {\kern 1pt}B, \Delta
\Rightarrow C}{G_1 \vert G_2 \vert \Gamma, {\kern 1pt}A\rightsquigarrow  B, \Pi, \Delta
\Rightarrow C}(\rightsquigarrow  _l )&\\
\cfrac{G_1 \vert \Gamma, A, \Delta \Rightarrow C\quad G_2 \vert \Gamma
, B, \Delta \Rightarrow C}{G_1 \vert G_2 \vert \Gamma, A\vee B, \Delta
\Rightarrow C}(\vee _l )&\\
\cfrac{G_1 \vert \Gamma \Rightarrow A\quad G_2 \vert \Gamma \Rightarrow
B}{G_1 \vert G_2 \vert \Gamma \Rightarrow A\wedge B}(\wedge _l )&\\
\cfrac{G\vert \Gamma, A, \Delta \Rightarrow C}{G\vert \Gamma, A\wedge B, \Delta
\Rightarrow C}(\wedge _{rr} )&\\
\cfrac{G\vert \Gamma, \Delta \Rightarrow A}{G\vert \Gamma, t, \Delta
\Rightarrow A}(t_l )
\end{align*}
\end{minipage}
\begin{minipage}[r]{0.49\linewidth}
\begin{align*}
\quad
\cfrac{G\vert \Gamma, A, B, \Delta \Rightarrow C}{G\vert \Gamma, A\odot
B, \Delta \Rightarrow C}(\odot _l )&\\
\cfrac{G\vert A, \Gamma \Rightarrow B}{G\vert \Gamma \Rightarrow A\to B}(\to
_r )&\\
\cfrac{G\vert \Gamma, A\Rightarrow B}{G\vert \Gamma \Rightarrow A\rightsquigarrow
B}(\rightsquigarrow_r )&\\
\cfrac{G\vert \Gamma \Rightarrow A}{G\vert \Gamma \Rightarrow A\vee B}(\vee
_{rr} )&\\
\cfrac{G\vert \Gamma \Rightarrow B}{G\vert \Gamma \Rightarrow A\vee B}(\vee
_{rl} )&\\
\cfrac{G\vert \Gamma, B, \Delta \Rightarrow C}{G\vert \Gamma, A\wedge B, \Delta
\Rightarrow C}(\wedge _{rl} )&\\
\end{align*}
\end{minipage}

\noindent\textbf{Cut Rule}
\[
\cfrac{G_1 \vert \Gamma, A, \Delta \Rightarrow B\quad G_2 \vert
\Pi \Rightarrow A}{G_1 \vert G_2 \vert \Gamma, \Pi, \Delta \Rightarrow
B}(CUT).
\]
\end{definition}

\begin{definition} ([3])
${\rm {\bf GpsUL}}^\ast $ is ${\rm {\bf GpsUL}}$ plus the weakly commutativity
rule
\[
\cfrac{G\vert \Gamma, \Delta \Rightarrow t}{G\vert \Delta, \Gamma \Rightarrow
t}(WCM).
\]
\end{definition}

\begin{definition}
${\rm {\bf GpsUL}}^{\ast D}$ is ${\rm {\bf GpsUL}}^\ast $ plus the density rule $\cfrac{G\vert \Pi \Rightarrow p\vert \Gamma, p, \Delta \Rightarrow B}{G\vert \Gamma, \Pi, \Delta \Rightarrow B}(D)$.
\end{definition}

\begin{lemma}
$G\equiv B\vee ((D\to B)\odot C\odot (C\to D)\odot A\to A)$ is not
a theorem in ${\rm {\bf HpsUL}}$.
\end{lemma}
\begin{proof}
Let $\mathcal{A}=(\{0, 1, 2, 3, 4, 5\}, \wedge, \vee, \odot, \to
, \rightsquigarrow, 3, 0, 5)_{ }$be an algebra,  where $x\wedge y=\min
(x, y)$,  $x\vee y=\max (x, y)$ for all $x,  y\in
\{0, 1, 2, 3, 4, 5\}$,  and the binary operations $\odot $,
$\to $ and $\rightsquigarrow$ are defined by the following tables (See [2]).
\scriptsize
\begin{center}
\begin{minipage}[b]{0.32\textwidth}
\begin{tabular}[t]{*{8}{|l}}\hline
$\odot$&0&1&2&3&4&5 \\\hline $0$&0&0&0&0&0&0 \\\hline
$1$&0&0&1&1&1&1
\\\hline 2&0&1&2&2&2&5 \\\hline 3&0&1&2&3&4&5 \\\hline
4&0&1&2&4&4&5
\\\hline
5&0&1&2&5&5&5
\\\hline
\end{tabular}
%\tabcaption{definition of $\odot$.}
\end{minipage}%
\begin{minipage}[b]{0.32\textwidth}
\begin{tabular}[t]{*{8}{|r}}\hline
$\rightarrow$&0&1&2&3&4&5 \\\hline $0$&5&5&5&5&5&5
\\\hline $1$&1&5&5&5&5&5
\\\hline 2&0&1&4&4&4&5 \\\hline 3&0&1&2&3&4&5 \\\hline
4&0&1&2&2&4&5
\\\hline
5&0&1&2&2&2&5
\\\hline
\end{tabular}
%\tabcaption{definition of $\rightarrow$.}
\end{minipage}%
\begin{minipage}[b]{0.32\textwidth}
\begin{tabular}[t]{*{8}{|r}}\hline
$\rightsquigarrow$&0&1&2&3&4&5 \\\hline $0$&5&5&5&5&5&5
\\\hline $1$&1&5&5&5&5&5
\\\hline 2&0&1&5&5&5&5 \\\hline 3&0&1&2&3&4&5 \\\hline
4&0&1&2&2&4&5
\\\hline
5&0&1&1&1&1&5
\\\hline
\end{tabular}
%\tabcaption{definition of $\rightsquigarrow$.}
\end{minipage}\\
\end{center}
\normalsize
By easy calculation,  we get that $\mathcal{A}$ is a linearly ordered ${\rm
{\bf HpsUL}}$-algebra,  where $ 0$ and $5$ are the least and the greatest
element of $\mathcal{A}$,  respectively,  and 3 is its unit. Let $v(A)=3,  v(B)=v(C)=2,  v(D)=5$. Then $v(G)=2\vee(2\odot 2\odot 5\odot 3\to 3)=2<3$. Hence $G$ is not a tautology in ${\rm {\bf HpsUL}}$. Therefore it is not a theorem in ${\rm {\bf
HpsUL}}$ by Theorem 9.27 in [1].
\end{proof}

\begin{theorem}
Cut-elimination doesn't holds for ${\rm {\bf GpsUL}}^\ast $.
\end{theorem}
\begin{proof}
$G\equiv \Rightarrow B\vee ((D\to B)\odot C\odot
(C\to D)\odot A\to A)$ is provable in ${\rm {\bf GpsUL}}^\ast $,  as shown in Figure
1.
\scriptsize
\[
\cfrac{\cfrac{\cfrac{\cfrac{A\Rightarrow A}{t, A\Rightarrow A}(t_l )\, \, \, \cfrac{\cfrac{\cfrac{{\kern
1pt}B\Rightarrow B{\kern 1pt}{\kern 1pt}{\kern 1pt}{\kern 1pt}{\kern
1pt}{\kern 1pt}{\kern 1pt}{\kern 1pt}{\kern 1pt}{\kern 1pt}{\kern 1pt}{\kern
1pt}{\kern 1pt}{\kern 1pt}{\kern 1pt}{\kern 1pt}{\kern 1pt}{\kern 1pt}{\kern
1pt}{\kern 1pt}{\kern 1pt}{\kern 1pt}{\kern 1pt}{\kern 1pt}\Rightarrow
t}{{\kern 1pt}\Rightarrow B{\kern 1pt}{\kern 1pt}\vert {\kern 1pt}{\kern
1pt}B\Rightarrow t}(COM)\, \, \, \cfrac{{\kern 1pt}C\Rightarrow C{\kern
1pt}{\kern 1pt}{\kern 1pt}{\kern 1pt}{\kern 1pt}{\kern 1pt}{\kern 1pt}{\kern
1pt}{\kern 1pt}{\kern 1pt}{\kern 1pt}{\kern 1pt}D\Rightarrow D}{{\kern
1pt}C, C\to D\Rightarrow D}(\to _l )}{{\kern 1pt}\Rightarrow B{\kern
1pt}{\kern 1pt}\vert {\kern 1pt}{\kern 1pt}C, C\to D, D\to B\Rightarrow t}(\to
_l )}{\Rightarrow B{\kern 1pt}{\kern 1pt}\vert {\kern 1pt}{\kern 1pt}D\to
B, C, C\to D\Rightarrow t}(WCM){\kern 1pt}{\kern 1pt}{\kern 1pt}{\kern
1pt}}{\Rightarrow B{\kern 1pt}{\kern 1pt}\vert {\kern 1pt}{\kern 1pt}D\to
B, C, C\to D, A\Rightarrow A}(CUT)}{\Rightarrow B{\kern 1pt}{\kern 1pt}\vert
{\kern 1pt}{\kern 1pt}\Rightarrow (D\to B)\odot C\odot (C\to D)\odot A\to
A}(\odot _l^\ast, \to _r )}{\Rightarrow B{\kern 1pt}\vee ((D\to B)\odot
C\odot (C\to D)\odot A\to A)}(\vee _{rr}, \vee _{rl}, EC).
\]
\normalsize
\begin{center}
Figure 1 A proof $\tau $ of $G$
\end{center}

Suppose that $G$ has a cut-free proof $\rho $. Then there exists no
occurrence of $t$ in $\rho $ by its subformula property. Thus there exists
no application of $(WCM){\kern 1pt}$ in $\rho $. Hence $G$ is a theorem of
${\rm {\bf GpsUL}}$,  which contradicts Lemma 2.4.
\end{proof}

\begin{remark}
 Following the construction given in the proof of Theorem
53 in [4],  $(CUT)$ in the figure 1 is eliminated by the following
derivation. However,  the application of $(WCM)$ in $\rho $ is invalid,  which
illustrates the reason why the cut-elimination theorem doesn't hold in
 ${\rm {\bf GpsUL}}^\ast$.
\[
\cfrac{\cfrac{\cfrac{\cfrac{\cfrac{{\kern 1pt}B\Rightarrow B{\kern 1pt}{\kern
1pt}{\kern 1pt}{\kern 1pt}{\kern 1pt}{\kern 1pt}{\kern 1pt}{\kern 1pt}{\kern
1pt}{\kern 1pt}{\kern 1pt}{\kern 1pt}{\kern 1pt}{\kern 1pt}{\kern 1pt}{\kern
1pt}{\kern 1pt}{\kern 1pt}{\kern 1pt}{\kern 1pt}{\kern 1pt}{\kern 1pt}{\kern
1pt}{\kern 1pt}A\Rightarrow A}{{\kern 1pt}\Rightarrow B{\kern 1pt}{\kern
1pt}\vert {\kern 1pt}{\kern 1pt}B, A\Rightarrow A}(COM){\kern 1pt}{\kern
1pt}{\kern 1pt}{\kern 1pt}{\kern 1pt}{\kern 1pt}{\kern 1pt}{\kern 1pt}{\kern
1pt}{\kern 1pt}{\kern 1pt}{\kern 1pt}{\kern 1pt}{\kern 1pt}{\kern 1pt}{\kern
1pt}{\kern 1pt}{\kern 1pt}{\kern 1pt}{\kern 1pt}{\kern 1pt}{\kern
1pt}\cfrac{{\kern 1pt}C\Rightarrow C{\kern 1pt}{\kern 1pt}{\kern 1pt}{\kern
1pt}{\kern 1pt}{\kern 1pt}{\kern 1pt}{\kern 1pt}{\kern 1pt}{\kern 1pt}{\kern
1pt}{\kern 1pt}D\Rightarrow D}{{\kern 1pt}C, C\to D\Rightarrow D}(\to _l
)}{{\kern 1pt}\Rightarrow B{\kern 1pt}{\kern 1pt}\vert {\kern 1pt}{\kern
1pt}C, C\to D, D\to B, A\Rightarrow A}(\to _l )}{\Rightarrow B{\kern 1pt}{\kern
1pt}\vert {\kern 1pt}{\kern 1pt}D\to B, C, C\to D, A\Rightarrow
A}(WCM)}{\Rightarrow B{\kern 1pt}{\kern 1pt}\vert {\kern 1pt}{\kern
1pt}\Rightarrow (D\to B)\odot C\odot (C\to D)\odot A\to A}(\odot _l^\ast
, \to _r )}{\Rightarrow B{\kern 1pt}\vee ((D\to B)\odot C\odot (C\to D)\odot
A\to A)}(\vee _{rr}, \vee _{rl}, EC)
\]
\begin{center}
Figure 2 A possible cut-free proof $\rho $ of $G$
\end{center}
\end{remark}

\begin{definition}
${{\rm {\bf GpsUL}}^{\ast\ast}} $ is constructed by replacing $(CUT)$ in ${\rm {\bf GpsUL}}^\ast$ with
\[\cfrac{G_1
\vert \Gamma, t, \Delta \Rightarrow A{\kern 1pt}{\kern 1pt}{\kern 1pt}{\kern
1pt}{\kern 1pt}{\kern 1pt}{\kern 1pt}G_2 \vert \Pi \Rightarrow t{\kern 1pt}{\kern
1pt}}{G_1 \vert G_2 \vert \Gamma, \Pi, \Delta \Rightarrow A}(WCT).\]
We call it the weakly cut rule and,   denote by $(WCT)$.
\end{definition}

\begin{theorem}
If $\vdash _{{{\rm {\bf GpsUL}}^{\ast}} } G$,  then $\vdash _{{{\rm {\bf GpsUL}}^{\ast\ast}}} G$.
\end{theorem}
\begin{proof}
It is proved by a procedure similar to that of Theorem 53 in
[4] and omitted.
\end{proof}

\begin{definition}([6])
${\rm {\bf GpsUL}}_{\Omega } $ is a restricted subsystem of ${\rm {\bf GpsUL}}^\ast$
such that

(i) $p$ is designated as the unique eigenvariable by which we means that it
does not be used to built up any formula containing logical connectives and
only used as a sequent-formula.

(ii) Each occurrence of $p$ in a hypersequent is assigned one unique
identification number $i$ in ${\rm {\bf GpsUL}}_{\Omega } $ and written as
$p_i $. Initial sequent $p\Rightarrow p$ of ${\rm {\bf GpsUL}}^{\ast}$ has the form
$p_i \Rightarrow p_i $ in ${\rm {\bf GpsUL}}_{\Omega }$.  $p$ doesn't
occur in $A,  \Gamma $ or $\Delta $ for each initial sequent $\Gamma, \bot
, \Delta \Rightarrow A$ or $\Gamma \Rightarrow \top $ in ${\rm {\bf GpsUL}}_{\Omega }$.

(iii) Each sequent $S$ of the form $\Gamma _0, p, \Gamma _1, \cdots, \Gamma
_{\lambda -1}, p, \Gamma _\lambda \Rightarrow A$ in ${\rm {\bf GpsUL}}^{\ast }$ has the form $\Gamma _0, p_{i_1 }, \Gamma _1, \cdots, \Gamma _{\lambda -1}
, p_{i_\lambda }, \Gamma _\lambda \Rightarrow A$ in ${\rm {\bf GpsUL}}_{\Omega }$,  where $p$ does not occur in $\Gamma _k $ for all
$0\leqslant k\leqslant \lambda$ and,  $i_k \ne i_l $ for all
$1\leqslant k<l\leqslant \lambda $. Define $v_l (S)=\{i_1, \cdots, i_\lambda
\}$,  $v_r (S)=\{j_1 \}$ if $A$ is an eigenvariable with the identification
number $j_1 $ and,  $v_r (S)=\emptyset $ if $A$ isn't an eigenvariable.

Let $G$ be a hypersequent of ${\rm {\bf GpsUL}}_{\Omega }$ in the form
$S_1 \vert \cdots \vert S_n $ then $v_l (S_k )\bigcap v_l (S_l )=\emptyset $
and $v_r (S_k )\bigcap v_r (S_l )=\emptyset $ for all $1\leqslant k<l\leqslant
n$. Define $v_l (G)=\bigcup _{k=1}^n v_l (S_k )$,  $v_r (G)=\bigcup _{k=1}^n v_r
(S_k )$.

(iv) A hypersequent $G$ of ${\rm {\bf GpsUL}}_{\Omega }$ is called closed
if $v_l (G)=v_r (G)$. Two hypersequents $G'$ and $G''$ of ${\rm {\bf GpsUL}}_{\Omega }$ are called disjoint if $v_l (G')\bigcap v_l
(G'')=\emptyset $, $v_l (G')\bigcap v_r (G'')=\emptyset $,  \\
$v_r (G')\bigcap v_l
(G'')=\emptyset $ and $v_r (G')\bigcap v_r (G'')=\emptyset $.  $G''$ is a copy
of $G'$ if they are disjoint and there exist two bijections $\sigma _l: v_l
(G')\to v_l (G'')$ and $\sigma _r: v_r (G')\to v_r (G'')$ such that $G''$
can be obtained by applying $\sigma _l $ to antecedents of sequents in $G'$
and $\sigma _r $ to succedents of sequents in $G'$.

(v) A hypersequent $G\vert G_1 \vert G_2 $ can be contracted as $G\vert G_1
$ in ${\rm {\bf GpsUL}}_{\Omega } $ under certain condition given in
Construction 3.15,  which we called the constraint external contraction rule and
denote by $\cfrac{G'\vert G_1 \vert G_2
}{G'\vert G_1 }(EC_\Omega )$.

(vi) $(EW)$ is forbidden in ${\rm {\bf GpsUL}}_{\Omega } $ and,  $(EC)$ and
$(CUT)$ are replaced with $(EC_\Omega )$ and $(WCT)$,   respectively.

(vii) Two rules $(\wedge _r )$ and $(\vee _l )$ of ${\rm {\bf GL}}$ are
replaced with $\cfrac{G_1 \vert \Gamma _1 \Rightarrow A\quad G_2 \vert \Gamma
_2 \Rightarrow B}{G_1 \vert G_2 \vert \Gamma _1 \Rightarrow A\wedge B\vert
\Gamma _2 \Rightarrow A\wedge B}(\wedge _{rw} )$ and
$\cfrac{G_1 \vert \Gamma _1, A, \Delta _1 \Rightarrow C_1 \quad G_2 \vert
\Gamma _2, B,\Delta _2 \Rightarrow C_2 }{G_1 \vert G_2 \vert \Gamma _1
, A\vee B, \Delta _1 \Rightarrow C_1 \vert \Gamma _2, A\vee B, \Delta _2
\Rightarrow C_2 }(\vee _{lw} )$ in ${\rm {\bf GpsUL}}_{\Omega }$,
respectively.

(viii) $G_1 \vert S_1$ and $G_2 \vert S_2 $ are closed and
disjoint  for each two-premise inference rule \\
$\cfrac{G_1\vert S_1\, \, \, G_2\vert S_2 }{G_1\vert G_2\vert H'}(II)$ of ${\rm {\bf GpsUL}}_{\Omega }$ and,  $G'\vert S'$ is closed
for each one-premise inference rule $\cfrac{G'\vert S'}{G'\vert S''}(I)$.
\end{definition}

\begin{proposition}
Let $\cfrac{G'\vert S'{\kern 1pt}{\kern 1pt}{\kern 1pt}{\kern 1pt}{\kern
1pt}}{G'\vert S''}(I)$ and $\cfrac{G_1 \vert S_1 {\kern 1pt}{\kern
1pt}{\kern 1pt}{\kern 1pt}{\kern 1pt}G_2 \vert S_2 }{G_1 \vert G_2 \vert
H'}(II)$ be inference rules of ${\rm {\bf GpsUL}}_{\Omega }$ then
$v_l
(G'\vert S'')=v_r (G'\vert S'')=v_r (G'\vert S')=v_l (G'\vert S')$ and $v_l
(G_1 \vert G_2 \vert H')=v_l (G_1 \vert S_1 )\bigcup v_l (G_2 \vert S_2 )=$
$v_r (G_1 \vert G_2 \vert H')=v_r (G_1 \vert S_1 )\bigcup v_r (G_2 \vert S_2
)$.
\end{proposition}
\begin{proof} Although $(WCT)$ makes $t$'s in its premises disappear in
its conclusion,  it has no effect on identification numbers of the
eigenvariable $p$ in a hypersequent because $t$ is a constant in ${\rm {\bf GpsUL}}_{\Omega }$ and distinguished from propositional variables.
\end{proof}

\begin{definition} [1]
Let $G$ be a closed hypersequent of ${\rm {\bf GpsUL}}_{\Omega }$ and $S\in G$. $[S]_G: =\bigcap \{H: S\in H\subseteq G, v_l
(H)=v_r (H)\}$ is called a minimal closed unit of $G$.
\end{definition}

\section{The generalized density rule $(\mathcal{D})$ for ${\rm {\bf GpsUL}}_{\rm
{\bf \Omega }} $}

In this section,   ${\rm {\bf GL}}_{\rm {\bf \Omega }}^{{\rm {\bf cf}}} $
is ${\rm {\bf G}}_{{\rm {\bf ps}}} {\rm {\bf UL}}{\Omega } $
without $(EC_\Omega )$.  Generally,  $A, B, C, \cdots $,  denote a formula other
than an eigenvariable $p_i $.

\begin{construction}
Given a proof $\tau ^\ast $ of $H\equiv G\vert \Gamma, p_j, \Delta
\Rightarrow p_j $ in ${\rm {\bf GL}}_{\rm {\bf \Omega }}^{{\rm {\bf cf}}} $,
let $Th_{\tau ^\ast } (p_j \Rightarrow p_j )=\left( {H_0, \cdots, H_n }
\right)$,  where $H_0 \equiv p_j \Rightarrow p_j $, $H_n \equiv H$. By $\Gamma
_k, p_j, \Delta _k \Rightarrow p_j $ we denote the sequent containing $p_j $
in $H_k $. Then $\Gamma _0 =\emptyset $,   $\Delta _0 =\emptyset $,  $\Gamma _n
=\Gamma $ and $\Delta _n =\Delta $. Hypersequents $\left\langle {H_k }
\right\rangle _j^- $,  $\left\langle {H_k } \right\rangle _j^+ $ and their
proofs $\langle \tau^\ast\rangle _{j}^{-} \left( {\left\langle {H_k } \right\rangle _j^-
} \right)$,  $\langle \tau^\ast\rangle _{j}^{+} \left( {\left\langle {H_k } \right\rangle
_j^+ } \right)$ are constructed inductively for all $0\leqslant
k\leqslant n$ in the following such that $\Gamma _k \Rightarrow t\in
\left\langle {H_k } \right\rangle _j^- $, $\Delta _k \Rightarrow t\in
\left\langle {H_k } \right\rangle _j^+ $,  and $\left\langle {H_k }
\right\rangle _j^+ \backslash \{\Delta _k \Rightarrow t\}\vert \left\langle
{H_k } \right\rangle _j^- \backslash \{\Gamma _k \Rightarrow t\}=H_k
\backslash \{\Gamma _k, p_j, \Delta _k \Rightarrow p_j \}$.

(i) $\left\langle {H_0 } \right\rangle _j^-: =\left\langle {H_0 }
\right\rangle _j^+: =\Rightarrow t$,  $\langle \tau^\ast\rangle _{j}^{-}\left( {\left\langle
{H_0 } \right\rangle _j^- } \right)$ and $\langle \tau^\ast\rangle _{j}^{+}\left(
{\left\langle {H_0 } \right\rangle _j^+ } \right)$ are built up with
$\Rightarrow t$.

(ii) Let $\cfrac{G'\vert S'\, \, \, G''\vert S''{\kern 1pt}{\kern
1pt}{\kern 1pt}{\kern 1pt}}{G'\vert G''\vert H'{\kern 1pt}}(II)$ (or
$\cfrac{G'\vert S'{\kern 1pt}{\kern 1pt}{\kern 1pt}{\kern 1pt}{\kern
1pt}{\kern 1pt}{\kern 1pt}{\kern 1pt}}{G'\vert S''{\kern 1pt}{\kern
1pt}{\kern 1pt}}(I))$ be in $\tau ^\ast $,  $H_k =G'\vert S'{\kern 1pt}$ and
$H_{k+1} =G'\vert G''\vert H'$ (accordingly $H_{k+1} =G'\vert S''{\kern
1pt}{\kern 1pt}$ for $(I))$ for some $0\leqslant k\leqslant n-1$. There are
three cases to be considered.

\textbf{Case 1} $S'=\Gamma _k, p_j, \Delta _k \Rightarrow p_j $.
If all focus formula(s) of $S'{\kern 1pt}$ is (are) contained in $\Gamma _k
$,
 \[\left\langle {H_{k+1} } \right\rangle _j^-: =\left( {\left\langle {H_k }
\right\rangle _j^- \backslash \{\Gamma _k \Rightarrow t\}} \right)\vert
G''\vert H'\backslash \{\Gamma _{k+1}, p_j, \Delta _{k{+}1}
\Rightarrow p_j \}\vert \Gamma _{k+1} \Rightarrow t\]
\[\left\langle {H_{k+1}
} \right\rangle _j^+: =\left\langle {H_k }
\right\rangle _j^+\]
(accordingly $\left\langle {H_{k+1} } \right\rangle
_j^- =\left\langle {H_k } \right\rangle _j^- \backslash \{\Gamma _k
\Rightarrow t\}\vert \Gamma _{k+1} \Rightarrow t$ for $(I)$)
and,  $\langle \tau^\ast\rangle _{j}^{-}\left( {\left\langle {H_{k+1} } \right\rangle _j^- }
\right)$ is constructed by combining the derivation $\langle \tau^\ast\rangle _{j}^{-}\left(
{\left\langle {H_k } \right\rangle _j^- } \right)$ and $\cfrac{\left\langle
{H_k } \right\rangle _j^-\, \, \, G''\vert S'{\kern
1pt}'{\kern 1pt}{\kern 1pt}{\kern 1pt}{\kern 1pt}{\kern 1pt}}{\left\langle
{H_{k+1} } \right\rangle _j^- {\kern 1pt}{\kern 1pt}}(II)$
(accordingly $\cfrac{\left\langle {H_k } \right\rangle _j^- {\kern 1pt}{\kern
1pt}{\kern 1pt}{\kern 1pt}}{\left\langle {H_{k+1} } \right\rangle _j^-
{\kern 1pt}{\kern 1pt}}(I){\kern 1pt}$ for $(I))$ and,  $\langle \tau^\ast\rangle _{j}^{+} \left( {\left\langle {H_{k+1} } \right\rangle _j^+ } \right)$ is
constructed by combining $\langle \tau^\ast\rangle _{j}^{+}\left( {\left\langle {H_k }
\right\rangle _j^+ } \right)$ and $\cfrac{\left\langle {H_k } \right\rangle
_j^+ {\kern 1pt}{\kern 1pt}{\kern 1pt}{\kern 1pt}}{\left\langle {H_{k+1} }
\right\rangle _j^+ {\kern 1pt}}(ID_\Omega ){\kern 1pt}$. The case of all
focus formula(s) of $S'$ contained in $\Delta _k $ is dealt with
by a procedure dual to above and omitted.

\textbf{Case 2} $S'\in \left\langle {H_k } \right\rangle _j^- $.
$\left\langle {H_{k+1} } \right\rangle _j^-: =\left( {\left\langle {H_k }
\right\rangle _j^- \backslash \{S'\}} \right)\vert G''\vert H'$ (accordingly\\
$\left\langle {H_{k+1} } \right\rangle _j^- =\left\langle {H_k }
\right\rangle _j^- \backslash \{S'\}\vert S''$ for $(I))$,
$\left\langle {H_{k+1} } \right\rangle _j^+: =\left\langle {H_k } \right\rangle _j^+ $ and $\langle \tau^\ast\rangle _{j}^{-}\left({\left\langle {H_{k+1} } \right\rangle _j^- } \right)$ is constructed by
combining the derivation
$\langle \tau^\ast\rangle _{j}^{-}\left( {\left\langle {H_k }
\right\rangle _j^- } \right)$ and $\cfrac{\left\langle {H_k } \right\rangle
_j^- \, \, \, G''\vert S'{\kern 1pt}'{\kern
1pt}{\kern 1pt}{\kern 1pt}{\kern 1pt}{\kern 1pt}}{\left\langle {H_{k+1} }
\right\rangle _j^- {\kern 1pt}{\kern 1pt}}(II)$
(accordingly$\cfrac{\left\langle {H_k } \right\rangle _j^- {\kern 1pt}{\kern
1pt}{\kern 1pt}{\kern 1pt}}{\left\langle {H_{k+1} } \right\rangle _j^-
{\kern 1pt}{\kern 1pt}}(I){\kern 1pt}$ for $(I))$ and,  $\langle \tau^\ast\rangle _{j}^{+}\left( {\left\langle {H_{k+1} }
\right\rangle _j^+ } \right)$ is constructed by combining $\langle \tau^\ast\rangle _{j}^{+}\left( {\left\langle {H_k } \right\rangle _j^+ } \right)$ and
$\cfrac{\left\langle {H_k } \right\rangle _j^+ {\kern 1pt}{\kern 1pt}{\kern
1pt}{\kern 1pt}}{\left\langle {H_{k+1} } \right\rangle _j^+ {\kern
1pt}}(ID_\Omega ){\kern 1pt}$.

\textbf{Case 3} $S'\in \left\langle {H_k } \right\rangle _j^+ $.
It is dealt with by a procedure dual to Case 2 and omitted.
\end{construction}

\begin{definition}
The manipulation described in Construction 3.1 is called the derivation-splitting operation when it is applied to a derivation and, the splitting operation when applied to a hypersequent.
\end{definition}

\begin{corollary}
Let $\vdash _{{\rm {\bf GL}}_{\rm {\bf \Omega }}^{{\rm {\bf cf}}} } G\vert
\Gamma, p_1, \Delta \Rightarrow p_1 $. Then there exist two hypersequents
$G_1 $ and $G_2 $ such that $G=G_1 \bigcup G_2 $,  $G_1 \bigcap G_2 =\emptyset $,
$\vdash _{{\rm {\bf GL}}_{\rm {\bf \Omega }}^{{\rm {\bf cf}}} } G_1 \vert
\Gamma \Rightarrow t$ and $\vdash _{{\rm {\bf GL}}_{\rm {\bf \Omega }}^{{\rm
{\bf cf}}} } G_2 \vert \Delta \Rightarrow t$.
\end{corollary}

\begin{construction}
Given a proof $\tau ^\ast $ of $H\equiv G\vert \Pi \Rightarrow p_j \vert
\Gamma, p_j, \Delta \Rightarrow A$ in ${\rm {\bf GL}}_{\rm {\bf \Omega
}}^{{\rm {\bf cf}}} $,  let $Th_{\tau ^\ast } (p_j \Rightarrow p_j )=(H_0
, \cdots, H_n )$,  where $H_0 \equiv p_j \Rightarrow p_j $ and $H_n \equiv H$.
Then there exists $1\leqslant m\leqslant n$ such that $H_m $ is in the form
$G'\vert \Pi '\Rightarrow p_j \vert \Gamma ', p_j, \Delta '\Rightarrow A'$
and $H_{m-1} $ is in the form $G''\vert \Gamma '', p_j, \Delta ''\Rightarrow
p_j $. A proof of $G\vert \Gamma, \Pi, \Delta \Rightarrow A$ in ${\rm {\bf
GL}}_{\rm {\bf \Omega }}^{{\rm {\bf cf}}} $ is constructed by induction on
$n-m$ as follows.

\begin{itemize}
\item For the base step,  let $n-m=0$. Then \\
$\cfrac{H_{n-1} \equiv G'\vert \Pi ', \Gamma ', p_j, \Delta ', \Pi '''\Rightarrow p_j {\kern 1pt}\quad G''\vert \Gamma '', \Pi '', \Delta ''\Rightarrow A}{H_n \equiv G'\vert G''\vert \Pi ', \Pi '', \Pi '''\Rightarrow p_j {\kern 1pt}\vert \Gamma '', \Gamma ', p_j, \Delta ', \Delta ''\Rightarrow A}(COM)\in \tau ^\ast $,  where $G'\vert G''=G$ and $\Pi ', \Pi '', \Pi '''=\Pi $ and $\Gamma '', \Gamma '=\Gamma $ and $\Delta ', \Delta ''=\Delta $. It follows from Corollary 3.3 that there exist $G_1 '$ and $G_2 '$ such that $G'=G_1 '\bigcup G_2 '$,  $G_1 '\bigcap G_2 '=\emptyset $,  $\vdash _{{\rm {\bf GL}}_{\rm {\bf \Omega }}^{{\rm {\bf cf}}} } G_1 '\vert \Pi ', \Gamma '\Rightarrow t$ and $\vdash _{{\rm {\bf GL}}_{\rm {\bf \Omega }}^{{\rm {\bf cf}}} } G_2 '\vert \Delta ', \Pi '''\Rightarrow t$. Then $G\vert \Gamma, \Pi, \Delta \Rightarrow A$ is proved as follows.

\footnotesize
$\boxed {\cfrac{\cfrac{\cfrac{\cfrac{G''\vert \Gamma '', \Pi '', \Delta ''\Rightarrow A{\kern 1pt}{\kern 1pt}{\kern 1pt}}{G''\vert \Gamma '', t, \Pi '', \Delta ''\Rightarrow A{\kern 1pt}{\kern 1pt}}(t_l ){\kern 1pt}{\kern 1pt}{\kern 1pt}{\kern 1pt}{\kern 1pt}{\kern 1pt}{\kern 1pt}{\kern 1pt}{\kern 1pt}{\kern 1pt}{\kern 1pt}{\kern 1pt}{\kern 1pt}{\kern 1pt}{\kern 1pt}\cfrac{G_1 '\vert \Pi ', \Gamma '\Rightarrow t}{G_1 '\vert \Gamma ', \Pi '\Rightarrow t}(WCM)}{G''\vert G_1 '\vert \Gamma '', \Gamma ', \Pi ', \Pi '', \Delta ''\Rightarrow A{\kern 1pt}}(WCT)}{G''\vert G_1 '\vert \Gamma '', \Gamma ', \Pi ', \Pi '', t, \Delta ''\Rightarrow A}(t_l ){\kern 1pt}{\kern 1pt}{\kern 1pt}{\kern 1pt}{\kern 1pt}{\kern 1pt}{\kern 1pt}{\kern 1pt}{\kern 1pt}{\kern 1pt}{\kern 1pt}{\kern 1pt}{\kern 1pt}{\kern 1pt}{\kern 1pt}\cfrac{G_2 '\vert \Delta ', \Pi '''\Rightarrow t}{G_2 '\vert \Pi ''', \Delta '\Rightarrow t}(WCM){\kern 1pt}{\kern 1pt}}{G''\vert G_1 '\vert G_2 '\vert \Gamma '', \Gamma ', \Pi ', \Pi '', \Pi ''', \Delta ', \Delta ''\Rightarrow A}(WCT)}$.
\normalsize
\item For the induction step,  let $n-m>0$. Then it is treated using applications of the induction hypothesis to the premise followed by an application of the relevant rule. For example,  let $\cfrac{H_{n-1} =G'\vert \Pi \Rightarrow p_j \vert \Sigma ', \Gamma '', p_j, \Delta '', \Sigma '''\Rightarrow A'{\kern 1pt}{\kern 1pt}{\kern 1pt}{\kern 1pt}{\kern 1pt}{\kern 1pt}{\kern 1pt}{\kern 1pt}{\kern 1pt}{\kern 1pt}{\kern 1pt}{\kern 1pt}{\kern 1pt}{\kern 1pt}{\kern 1pt}{\kern 1pt}{\kern 1pt}{\kern 1pt}{\kern 1pt}{\kern 1pt}{\kern 1pt}{\kern 1pt}{\kern 1pt}G''\vert \Gamma ', \Sigma '', \Delta '\Rightarrow A{\kern 1pt}}{H_n =G'\vert \Pi \Rightarrow p_j \vert \Sigma ', \Sigma '', \Sigma '''\Rightarrow A'{\kern 1pt}\vert G''\vert \Gamma ', \Gamma '', p_j, \Delta '', \Delta '\Rightarrow A{\kern 1pt}}(COM)\in \tau ^\ast $,  where $G'{\kern 1pt}\vert G''\vert \Sigma ', \Sigma '', \Sigma '''\Rightarrow A'=G$ and $\Gamma ', \Gamma ''=\Gamma $ and $\Delta '', \Delta '=\Delta $.
By the induction hypothesis we obtain a derivation of $G\vert \Gamma, \Pi
, \Delta \Rightarrow A$:
\[
\cfrac{G'\vert \Sigma ', \Gamma '', \Pi, \Delta '', \Sigma '''\Rightarrow
A'{\kern 1pt}{\kern 1pt}{\kern 1pt}{\kern 1pt}{\kern 1pt}{\kern 1pt}{\kern
1pt}{\kern 1pt}{\kern 1pt}{\kern 1pt}{\kern 1pt}{\kern 1pt}{\kern 1pt}{\kern
1pt}{\kern 1pt}{\kern 1pt}{\kern 1pt}{\kern 1pt}{\kern 1pt}{\kern 1pt}{\kern
1pt}{\kern 1pt}{\kern 1pt}G''\vert \Gamma ', \Sigma '', \Delta '\Rightarrow
A{\kern 1pt}}{G'\vert \Sigma ', \Sigma '', \Sigma '''\Rightarrow A'{\kern
1pt}\vert G''\vert \Gamma ', \Gamma '', \Pi, \Delta '', \Delta '\Rightarrow
A{\kern 1pt}}(COM).
\]
\end{itemize}
\end{construction}
\begin{definition}
The manipulation described in Construction 3.4 is called the
derivation-splicing operation when it is applied to a derivation and, the splicing operation when applied to a hypersequent.
\end{definition}

\begin{corollary}
If $\vdash _{{\rm {\bf GL}}_{\rm {\bf \Omega }}^{{\rm {\bf cf}}} } G\vert
\Pi \Rightarrow p_j \vert \Gamma, p_j, \Delta \Rightarrow A$,  then $\vdash
_{{\rm {\bf GL}}_{\rm {\bf \Omega }}^{{\rm {\bf cf}}} } G\vert \Gamma, \Pi
, \Delta \Rightarrow A$.
\end{corollary}
\begin{definition}
(i) Let $\vdash _{{\rm {\bf GL}}_{\rm {\bf \Omega }}^{{\rm {\bf cf}}} }
H\equiv G\vert \Gamma, p_j, \Delta \Rightarrow p_j $. Define $\left\langle H
\right\rangle _j^- =G_1 \vert \Gamma \Rightarrow t$, $_{ }\left\langle H
\right\rangle _j^+ =G_2 \vert \Delta \Rightarrow t$ and $D_j (H)=\{G_1 \vert
\Gamma \Rightarrow t,  G_2 \vert \Delta \Rightarrow t\}$,  where,  $G_1 $ and $G_2 $ are determined by Corollary 3.3.

(ii) Let $\vdash _{{\rm {\bf GL}}_{\rm {\bf \Omega }}^{{\rm {\bf cf}}} }
H\equiv G\vert \Pi \Rightarrow p_j \vert \Gamma, p_j, \Delta \Rightarrow A$.
Define$D_j (H)=\{G\vert \Gamma, \Pi, \Delta \Rightarrow A\}=\left\langle H
\right\rangle _j $.

(iii) Let $\vdash _{{\rm {\bf GL}}_{\rm {\bf \Omega }}^{{\rm {\bf cf}}} }
G$. $D_j (G)=\{G\}$ if $p_j $ does not occur in $G$.

(iv) Let $\vdash _{{\rm {\bf GL}}_{\rm {\bf \Omega }}^{{\rm {\bf cf}}} } G_i
$ for all $1\leqslant i\leqslant n$. Define $D_j (\{G_1, \cdots, G_n \})=D_j
(G_1 )\bigcup \cdots \bigcup D_j (G_n )$.

(v) Let $\vdash _{{\rm {\bf GL}}_{\rm {\bf \Omega }}^{{\rm {\bf cf}}} } G$
and $K=\{1, \cdots,  n \}\subseteq v(G)$. Define $D_K (G)=D_n (\cdots D_2
(D_1 (G))\cdots )$. Especially,  define $\mathcal{D}(G)=D_{v_l (G)} (G)$.
\end{definition}
\begin{theorem}
Let $\vdash _{{\rm {\bf GL}}_{\rm {\bf \Omega }}^{{\rm {\bf cf}}} } G$. Then
$\vdash _{{\rm {\bf GL}}_{\rm {\bf \Omega }}^{{\rm {\bf cf}}} } H$  for all
$H\in \mathcal{D}(G)$.
\end{theorem}
\begin{proof}Immediately from Corollary 3.3,  Corollary 3.6 and Definition
3.7.
\end{proof}

\begin{lemma}
Let $G'$ be a minimal closed unit of $G\vert G'$. Then $G'$ has the form
$\Gamma \Rightarrow A\vert \Gamma _{i_2 } \Rightarrow p_{i_2 } \vert \cdots
\vert \Gamma _{i_n } \Rightarrow p_{i_n } $ if there exists one sequent
$\Gamma \Rightarrow A\in G'$ such that $A$ is not an eigenvariable otherwise
$G'$ has the form $\Gamma _{i_1 } \Rightarrow p_{i_1 } \vert \cdots \vert
\Gamma _{i_n } \Rightarrow p_{i_n } $.
\end{lemma}

\begin{proof} Define $G_1 =\Gamma \Rightarrow A$ in Construction 5.2 in
[6]. Then $\emptyset =v_r (G_1 )\subseteq v_l (G_1 )$.  Suppose that $G_k
$ is constructed such that $v_r (G_k )\subseteq v_l (G_k )$. If $v_l (G_k
)=v_r (G_k )$,  the procedure terminates and $n: =k$,  otherwise $v_l (G_k
)\backslash v_r (G_k )\ne \emptyset $ and define $i_{k+1} $ to be an
identification number in $v_l (G_k )\backslash v_r (G_k )$. Then there
exists $\Gamma _{i_{k+1} } \Rightarrow p_{i_{k+1} } \in G\backslash G_k $ by
$v_l (G)=v_r (G)$ and,  define
$G_{k+1} =G_k \vert \Gamma _{i_{k+1} } \Rightarrow p_{i_{k+1} }. $ Thus $v_r
(G_{k+1} )=v_r (G_k )\bigcup \{i_{k+1} \}\subseteq v_l (G_k )\subseteq v_l
(G_{k+1} )$. Hence there exists a sequence $i_2, \cdots,  i_n$ of
identification numbers$_{ }$such that $v_r (G_k )\subseteq v_l (G_k )$ for
all $1\leqslant k\leqslant n$\textbf{,  }where $G_1 =\Gamma
\Rightarrow A$,  $G_k =\Gamma \Rightarrow A\vert \Gamma _{i_2 } \Rightarrow
p_{i_2 } \vert \cdots \vert \Gamma _{i_k } \Rightarrow p_{i_k } $ for all
$2\leqslant k\leqslant n$. Therefore $G'$ has the form $\Gamma \Rightarrow
A\vert \Gamma _{i_2 } \Rightarrow p_{i_2 } \vert \cdots \vert \Gamma _{i_n }
\Rightarrow p_{i_n } $.
\end{proof}

\begin{definition}
Let $G'$ be a minimal closed unit of $G\vert G'$. $G'$ is a splicing unit if
it has the form $\Gamma \Rightarrow A\vert \Gamma _{i_2 } \Rightarrow p_{i_2
} \vert \cdots \vert \Gamma _{i_n } \Rightarrow p_{i_n } $. $G'$ is a
splitting unit if it has the form $\Gamma _{i_1 } \Rightarrow p_{i_1 } \vert
\cdots \vert \Gamma _{i_n } \Rightarrow p_{i_n } $.
\end{definition}
\begin{lemma}
Let $G'$ be a splicing unit of $G\vert G'$ in the form $\Gamma \Rightarrow
A\vert \Gamma _{i_2 } \Rightarrow p_{i_2 } \vert \cdots \vert \Gamma _{i_n }
\Rightarrow p_{i_n } $ and $K=\{i_2, \cdots, i_n\}$.  Then $\left| {D_K
(G\vert G')} \right|=1$.
\end{lemma}

\begin{proof} By the construction in the proof of Lemma 3.9,  $i_k \in v_l
(G_{k-1} )$ for all $2\leqslant k\leqslant n$. Then $p_{i_2 } \in \Gamma $
and $D_{i_2 } (G\vert G')=G\vert \Gamma [\Gamma _{i_2 } ]\Rightarrow A\vert
\Gamma _{i_3 } \Rightarrow p_{i_3 } \vert \cdots \vert \Gamma _{i_n }
\Rightarrow p_{i_n } $,  where $\Gamma [\Gamma _{i_2 } ]_{ }$is obtained by
replacing $p_{i_2 } $ in $\Gamma $ with $\Gamma _{i_2 } $. Then $p_{i_3 }
\in \Gamma [\Gamma _{i_2 } ]$. Repeatedly,  we get $D_{i_2 \cdots i_n }
(G\vert G')=D_K (G\vert G')=G\vert \Gamma [\Gamma _{i_2 } ]\cdots [\Gamma
_{i_n } ]\Rightarrow A$.
\end{proof}

This shows that $D_K (G\vert G')$ is constructed by repeatedly applying
splicing operations.

\begin{definition}
Let $G'$ be a minimal closed unit of $G\vert G'$.  Define $V_{G'} =v(G')$,\\
  $E_{G'}
=\{(i, j)\vert \Gamma, p_i, \Delta \Rightarrow p_j \in G'\}$ and,  $j$ is
called the child node of $i$ for all $(i, j)\in E_{G'} $. We call $\Omega
_{G'} =(V_{G'}, E_{G'} )$ the $\Omega $-graph of $G'$.
\end{definition}

Let $G'$ be a splitting unit of $G\vert G'$ in the form $\Gamma _1
\Rightarrow p_1 \vert \cdots \vert \Gamma _n \Rightarrow p_n $. Then each
node of $\Omega _{G'} $ has one and only one child node. Thus there exists
one cycle in $\Omega _{G'} $ by $\left| {V_{G'} } \right|=n<\infty $. Assume
that,  without loss of generality,  $(1, 2), (2, 3), \cdots, (i, 1)$ is the cycle
of $\Omega _{G'} $. Then $p_1 \in \Gamma _2 $, $p_2 \in \Gamma _3 $, $\cdots
, p_{i-1} \in \Gamma _i $ and $p_i \in \Gamma _1 $. Thus $D_{i\cdots 2}
(G\vert G')=G\vert \Gamma _1 [\Gamma _i ][\Gamma _{i-1} ]\cdots [\Gamma _2
]\Rightarrow p_1 $ is in the form $G\vert \Gamma ', p_1, \Delta '\Rightarrow
p_1 $. By a suitable permutation $\sigma $ of $i+1, \cdots, n$,  we get
$D_{i\cdots 2\sigma (i+1\cdots n)} (G\vert G')=G\vert \Gamma _1 [\Gamma _i
][\Gamma _{i-1} ]\cdots [\Gamma _2 ][\Gamma _{\sigma (i+1)} ]\cdots [\Gamma
_{\sigma (n)} ]\Rightarrow p_1 =G\vert \Gamma, p_1, \Delta \Rightarrow p_1
$. This process also shows that there exists only one cycle in $\Omega _{G'}
$. Then we introduce the following definition.

\begin{definition}
(i) $\Gamma _j \Rightarrow p_j $ is called a splitting sequent of $G'$
and $p_j $ its corresponding splitting variable for all $1\leqslant
j\leqslant i$.

(ii) Let $K=\{1, 2, \cdots, n\}$ and $D_1 (G\vert \Gamma,  p_1,  \Delta
\Rightarrow p_1)=\{G_1 \vert \Gamma \Rightarrow t,  G_2 \vert \Delta
\Rightarrow t\}$.  Define
 $\left\langle {G\vert G'} \right\rangle _K^- =G_1 \vert \Gamma
\Rightarrow t$,  $\left\langle {G\vert G'} \right\rangle _K^+ =G_2 \vert
\Delta \Rightarrow t$ and $D_K (G\vert G')=\{\left\langle {G\vert G'}
\right\rangle _K^+, \left\langle {G\vert G'} \right\rangle _K^- \}$.
\end{definition}

\begin{lemma}
If $G'$ be a splitting unit of $G\vert G'$,  $K=v(G')$ and $k$ be a splitting
variable of $G'$. Then $D_{K\backslash \{k\}} (G\vert G')$ is constructed by
repeatedly applying splicing operations and only the last operation $D_k $
is a splitting operation.
\end{lemma}

\begin{construction}{(\bf The constrained external contraction rule)}\\
Let $H\equiv G'\vert \left\{ {\left[ S \right]_H } \right\}_1 \vert \left\{
{\left[ S \right]_H } \right\}_2 $,  $\left\{ {\left[ S \right]_H }
\right\}_1 $ and $\left\{ {\left[ S \right]_H } \right\}_2 $ be two copies of
a minimal closed unit $\left[ S \right]_H $,  where we put two copies  into $\{\}_1 $ and $\{\}_2 $ in order to distinguish them.
For any splitting unit $\left[ {S'} \right]_H \subseteq G'$,   $\left\{ {\left[ S \right]_H }
\right\}_1 \vert \left\{ {\left[ S \right]_H } \right\}_2 \subseteq
\left\langle H \right\rangle _K^- $ or $\left\{ {\left[ S \right]_H }
\right\}_1 \vert \left\{ {\left[ S \right]_H } \right\}_2 \subseteq
\left\langle H \right\rangle _K^+ $, where $K=v(\left[ {S'} \right]_H )$. Then 
$G''\vert \left\{ {\left[ S \right]_H } \right\}_1 $ is constructed
by cutting off $\left\{ {\left[ S \right]_H } \right\}_2 $  and some sequents in $G'$  as follows.

(i)  If  $\left\{ {\left[ S \right]_H } \right\}_1 $ and $\left\{ {\left[ S
\right]_H } \right\}_2 $ are two splicing units, then $G'':=G'$;

(ii) If $\left\{ {\left[ S \right]_H } \right\}_1 $ and $\left\{ {\left[ S
\right]_H } \right\}_2 $ are two splitting units and,  $k$,  $k'$ their
splitting variables,  respectively,  $K=v(\left\{ {\left[ S \right]_H
} \right\}_1 )$,  $K'=v(\left\{ {\left[ S \right]_H } \right\}_2 )$,  $D_{K\backslash \{k\}} (\left\{{\left[ S \right]_H } \right\}_1 )=
\Gamma, p_k, \Delta \Rightarrow p_k $,\\
$D_{K'\backslash \{k'\}} (\left\{ {\left[ S \right]_H } \right\}_2 )=\Gamma
, p_{k'}, \Delta \Rightarrow p_{k'} $,
$D_{K\bigcup K'} (H)=\{G_1 '\vert \Gamma \Rightarrow t\vert \Gamma \Rightarrow t, G_2 '\vert \Delta \Rightarrow t, G_2 ''\vert \Delta \Rightarrow t\}$  or
$D_{K\bigcup K'} (H)=\{G_1 '\vert \Delta \Rightarrow t\vert \Delta \Rightarrow
t, G_2 '\vert \Gamma \Rightarrow t, G_2 ''\vert \Gamma \Rightarrow t\}$,  where
$G_1 '\bigcup G_2 '\bigcup G_2 ''=G'$ and $G_2''$ is a copy of $G_2'$.
Then  $G'':=G'\backslash{G_2''}$.

The above operation is called the constrained external contraction rule,
denoted by $\langle EC_\Omega^\ast \rangle$ and written as $\cfrac{G'\vert \left\{ {\left[ S
\right]_H } \right\}_1 \vert \left\{ {\left[ S \right]_H } \right\}_2
}{G''\vert \left\{ {\left[ S \right]_H } \right\}_1 }\langle EC_\Omega^\ast\rangle$.
\end{construction}

\begin{lemma}
If $\vdash _{{\rm {\bf GL}}_{\rm {\bf \Omega }}^{{\rm {\bf cf}}} } H$ as above. Then $\vdash _{{\rm {\bf GpsUL}}_{\rm {\bf \Omega }}} H'$  for all
$H'\in \mathcal{D}(G''\vert \left\{ {\left[ S \right]_H } \right\}_1 )$.
\end{lemma}

\section{Density elimination for ${\rm {\bf GpsUL}}^\ast $}

In this section,  we adapt the separation algorithm of branches in [6] to
${\rm {\bf GpsUL}}^\ast $ and prove the following theorem.

\begin{theorem}
Density elimination holds for ${\rm {\bf GpsUL}}^\ast $.
\end{theorem}

The proof of Theorem 4.1 runs as follows. It is sufficient to prove
that the following strong density rule
\[
\cfrac{G_0 \equiv G'\vert \left\{ {\Gamma _i, p, \Delta _i \Rightarrow A_i }
\right\}_{i=1\cdots n} \vert \left\{ {\Pi _j \Rightarrow p{\kern 1pt}}
\right\}_{j=1\cdots m} }{\mathcal{D}_0 \left( {G_0 } \right)\equiv G'\vert
\{\Gamma _i, \Pi _j, \Delta _i \Rightarrow A_i \}{\kern 1pt}_{i=1\cdots
n;j=1\cdots m} }\left( {\mathcal{D}_0 } \right)
\]
is admissible in ${\rm {\bf GpsUL}}^\ast $,  where $n, m\geqslant 1$,  $p$ does
not occur in $G', \Gamma _i, \Delta _i, {\kern 1pt}A_i, \Pi _j $ for all
$1\leqslant i\leqslant n$,  $1\leqslant j\leqslant m$.

Let $\tau $ be a proof of $G_0 $ in ${\rm {\bf GpsUL}}^{\ast \ast }$ by
Theorem 2.8. Starting with $\tau $,  we construct a proof $\tau ^\ast $ of
$G\vert G^\ast $ in ${\rm {\bf GL}}_{\rm {\bf \Omega }}^{{\rm {\bf cf}}}
$ by a preprocessing of $\tau $ described in Section 4 in [6].

In Step 1 of preprocessing of $\tau $,  a proof $\tau '$ is constructed by
replacing inductively all applications of $(\wedge _r )$ and $(\vee _l )$ in
$\tau $ with $(\wedge _{rw} )$ and $(\vee _{lw} )$ followed by an
application of $(EC)$,  respectively. In Step 2,  a proof $\tau ''$ is
constructed by converting all $\cfrac{G_i '''\vert \{S_i^c \}^{m_i '}}{G_i
'''\vert S_i^c }(EC^\ast )\in \tau '$ into $\cfrac{G_i ''\vert \{S_i^c
\}^{m_i '}}{G_i ''\vert \{S_i^c \}^{m_i '}}(ID_\Omega )$,  where $G_i
'''\subseteq G_i ''$. In Step 3,  a proof $\tau '''$ is constructed by
converting $\cfrac{G'}{G'\vert S'}(EW)\in \tau ''$ into
$\cfrac{G''}{G''}(ID_\Omega )$,  where $G''\subseteq G'$. In Step 4,  a proof
$\tau ''''$ is constructed by replacing some $G'\vert \Gamma ', p, \Delta
'\Rightarrow A'\in \tau '''$ (or $G'\vert \Gamma '\Rightarrow p\in \tau '''$
) with $G'\vert \Gamma ', \top, \Delta '\Rightarrow A'$ (or $G'\vert \Gamma
'\Rightarrow \bot )$. In Step 5,  a proof $\tau ^\ast $ is constructed by
assigning the unique identification number to each occurrence of $p$ in $\tau
''''$.  Let $H_i^c =G_i '\vert \{S_i^c \}^{m_i }$ denote the unique node of
$\tau ^\ast $ such that $H_i^c \leqslant G_i ''\vert \{S_i^c \}^{m_i }$and
$S_i^c $ is the focus sequent of $H_i^c $ in $\tau ^\ast $. We call $H_i^c
$,  $S_i^c $ the $i$-th $(pEC)$-node of $\tau ^\ast $ and $(pEC)$-sequent,
respectively. If we ignore the replacements from Step 4,  each sequent of $G$
is a copy of some sequent of $G_0 $ and,  each sequent of $G^\ast $ is a copy
of some contraction sequent in $\tau '$.

Now,  starting with $G\vert G^\ast $ and its proof $\tau ^\ast $,  we
construct a proof $\tau ^{\medstar}$ of $G^{\medstar}$  in ${\rm {\bf GpsUL}}_\Omega $ such
that each sequent of $G^{\medstar}$  is a copy of some sequent of $G$. Then $\vdash
_{{\rm {\bf GpsUL}}_\Omega } \mathcal{D}(G^{\medstar})$ by Theorem 3.8 and Lemma 3.16.  Then $\vdash _{{\rm {\bf GpsUL}}^\ast } \mathcal{D}_0
(G_0 )$ by Lemma 9.1 in [6].

In [6],  $G^{\medstar}$  is constructed by eliminating $(pEC)$-sequents in $G\vert
G^\ast $ one by one. In order to control the process,  we introduce the set
$I=\{H_{i_1 }^c, \cdots, H_{i_m }^c \}$ of maximal $(pEC)$-nodes of $\tau
^\ast $ (See Definition 4.2) and the set ${\rm {\bf I}}$ of the branches
relative to $I$ and construct $G_{\rm {\bf I}}^{\medstar}$ such that $G_{\rm {\bf
I}}^{\medstar}$ doesn't contain the contraction sequents lower than any node in $I$, i.e.,  $S_j^c \in G_{\rm {\bf I}}^{\medstar}$ implies $H_j^c \vert \vert H_i^c $ for
all $H_i^c \in I$. The procedure is called the separation algorithm of
branches in [6].

The problem we encounter in ${\rm {\bf GpsUL}}_\Omega $ is that Lemma 7.11
of [6]  doesn't hold because new derivation-splitting operations
make the conclusion of $(\mathcal{D})$-rule to be a set of hypersequents
rather than one hypersequent. Then $G_{\ddagger}^{m_{q'} } $ generally can't be contracted to $G_{\ddagger} $ in Step 2 of Stage 1 in Main algorithm in [6] and,  $\{G_{{\rm {\bf
I}}_{l\backslash r} }^{\medstar}\}^{m_{q'} }$ can't be contracted to $G_{{\rm {\bf
I}}_{l\backslash r} }^{\medstar}$ in Step 2 of Stage 2. Furthermore,  we sometimes
can't construct some branches to $I$ in ${\rm {\bf GpsUL}}_\Omega $ before
we construct $\tau _{\bf I}^{\medstar}$. Therefore we have to introduce a new induction
strategy for ${\rm{\bf GpsUL}}_\Omega $ and don't perform the induction on the number of branches. First we give some primary definitions and lemmas.

\begin{definition}
A $(pEC)$-node $H_i^c $ is maximal if no other $(pEC)$-node is higher than $H_i^c $.
Define $I_0 $ to be the set of maximal $(pEC)$-nodes in $\tau ^\ast $. A
nonempty subset $I$ of $I_0 $ is complete if $I$ contains all maximal
$(pEC)$-nodes higher than or equal to the intersection node $H_I^V $ of $I$.
Define $H_I^V =H_i^c $ if $I=\{H_i^c \}$, i.e.,  the intersection node of a
single node is itself.
\end{definition}

\begin{proposition}
(i) $H_i^c \parallel H_j^c $ for all $i\ne j$,  $H_i^c, H_j^c \in I_0 $.

(ii) Let $I$ be complete and $H_j^c \geqslant H_I^V $. Then $H_j^c \leqslant
H_i^c $ for some $H_i^c \in I$.

(iii) $I_0 $ is complete and $\{H_i^c \}$ is complete for all $H_i^c \in I_0
$.

(iv) If $I\subseteq I_0 $ is complete and $\left| I \right|>1$,  then $I_l $
and $I_r $ are complete,  where $I_l $ and $I_r $ denote the sets of all
maximal $(pEC)$-nodes in the left subtree and right subtree of $\tau ^\ast
(H_I^V )$,  respectively.

(v) If $I_1, I_2 \subseteq I_0 $ are complete,  then $I_1 \subseteq I_2
$, $I_2 \subseteq I_1 $ or $I_1 \bigcap I_2 =\emptyset $.
\end{proposition}

\begin{proof} (v) $I_1 \subseteq I_2 $, $I_2 \subseteq I_1 $ or $I_1 \bigcap I_2
=\emptyset $ holds by $H_{I_2 }^V \leqslant H_{I_1 }^V $, $H_{I_1 }^V
\leqslant H_{I_2 }^V $ or $H_{I_2 }^V \parallel H_{I_1 }^V $,  respectively.
\end{proof}

\begin{definition}
A labeled binary tree $\rho $ is constructed inductively by the following
operations.

(i) The root of $\rho $ is labeled by $I_0 $ and leaves labeled $\{H_i^c
\}\subseteq I_0 $.

(ii) If an inner node is labeled by $I$,  then its parent nodes are labeled
by $I_l $ and $I_r $,  where $I_l $ and $I_r $ are defined in Proposition 4.3
(iv).
\end{definition}

\begin{definition}
We define the height $o(I)$ of $I\in \rho $ by letting $o(I)=1$ for each
leave $I\in \rho $ and,  $o(I)=\max \{o(I_l ), o(I_r )\}+1$ for any non-leaf
node.
\end{definition}

Note that in Lemma 7.11 in [6] only uniqueness of $G_{H_1: G_2 }^{\medstar(J)}
\vert \widehat{S_2 }$ in $G_{H_{i_k }^c }^{\medstar}$ doesn't hold in ${\rm {\bf
GpsUL}}_\Omega $ and the following lemma holds in ${\rm {\bf GpsUL}}_\Omega
$.

\begin{lemma}
Let $\cfrac{G_1 \vert S_1 {\kern 1pt}{\kern 1pt}{\kern 1pt}{\kern 1pt}{\kern
1pt}{\kern 1pt}{\kern 1pt}{\kern 1pt}{\kern 1pt}G_2 \vert S_2
}{H_1 \equiv G_1 \vert G_2 \vert H''}(II)\in \tau ^\ast $,  $\tau _{G_b \vert
S_j^c }^\ast \in \tau _{H_i^c }^{\medstar}$,  $\cfrac{G_b \vert \left\langle {G_1
\vert S_1 } \right\rangle _{S_j^c } {\kern 1pt}{\kern 1pt}{\kern 1pt}{\kern
1pt}{\kern 1pt}G_2 \vert S_2 }{H_2 \equiv G_b \vert \left\langle {G_1 } \right\rangle
_{S_j^c } \vert G_2 \vert H''}(II)\in \tau _{G_b \vert S_j^c }^\ast $. Then
$H''$ is separable in $\tau _{H_i^c }^{\medstar(J)} $ and there are some copies of
$G_{H_1: G_2 }^{\medstar(J)} \vert \widehat{S_2 }$ in $G_{H_i^c }^{\medstar}$.
\end{lemma}

\begin{lemma}{\bf (New main algorithm for ${\rm {\bf GpsUL}}_\Omega )$}
Let $I$ be a complete subset of $I_0 $ and $\overline I =\{H_i^c: H_i^c
\leqslant H_j^c \, \, for\, \, some\, \, \, H_j^c \in I\}$.
Then there exist one close hypersequent $G_I^{\medstar}\subseteq _c G\vert G^\ast $
and its derivation $\tau _I^{\medstar}$ in ${\rm {\bf GpsUL}}_\Omega $ such that

(i) $\tau _I^{\medstar}$ is constructed by initial hypersequent
$\cfrac{\underline{{\kern 1pt}{\kern 1pt}{\kern 1pt}{\kern 1pt}{\kern
1pt}{\kern 1pt}{\kern 1pt}{\kern 1pt}{\kern 1pt}{\kern 1pt}{\kern 1pt}{\kern
1pt}{\kern 1pt}{\kern 1pt}{\kern 1pt}{\kern 1pt}{\kern 1pt}{\kern 1pt}{\kern
1pt}{\kern 1pt}{\kern 1pt}{\kern 1pt}{\kern 1pt}}}{G\vert G^\ast
}\left\langle {\tau ^\ast } \right\rangle $,  the fully constraint
contraction rules of the form $\cfrac{\underline{{\kern 1pt}{\kern 1pt}{\kern
1pt}{\kern 1pt}G_2 {\kern 1pt}{\kern
1pt}{\kern 1pt}{\kern 1pt}}}{G_1
}\left\langle {EC_\Omega ^\ast } \right\rangle $ and elimination rule of the
form \[\cfrac{\underline{G_{b_1} \vert
S_{j_1 }^c \, \, \, G_{b_2} \vert S_{j_2 }^c \, \, \, \cdots \, \, \, G_{b_w} \vert S_{j_w }^c }}{G_{{\rm {\bf I}}_{\rm {\bf j}} }^\ast =\left\{ {G_{b_k} }
\right\}_{k=1}^w \vert G_{\mathcal{I}_{\rm {\bf j}} }^\ast }\left\langle
{\tau _{{\rm {\bf I}}_{\rm {\bf j}} }^\ast } \right\rangle, \]
where
$1\leqslant w\leqslant |I|,   H_{j_k }^c\leftrightsquigarrow H_{j_l }^c $ for all $1\leqslant k<l\leqslant w$,   $I_{\rm {\bf j}} =\left\{ {H_{j_1 }^c, \cdots
, {\kern 1pt}{\kern 1pt}{\kern 1pt}{\kern 1pt}H_{j_w }^c {\kern 1pt}}
\right\}\subseteq \overline I $,  $\mathcal{I}_{\rm {\bf j}} =\{S_{j_1 }^c
, S_{j_2 }^c, \cdots, S_{j_w }^c \}$,  ${\rm {\bf
I}}_{\rm {\bf j}} =\{G_{b_1} \vert S_{j_1 }^c, G_{b_2} \vert S_{j_2 }^c
, \cdots, G_{b_w} \vert S_{j_w}^c \}$,  $G_{b_k} \vert S_{j_k }^c $ is closed for all $1\leqslant k\leqslant w$. Then $H_i^c \not {\leqslant }H_j^c $ for each $S_j^c \in
G_{\mathcal{I}_{\rm {\bf j}} }^\ast $ and $H_i^c \in I$.

(ii) For all $H\in \overline \tau _{I}^{\medstar} $,  let
\[
\partial _{\tau_{I}^{\medstar}} (H): =\left\{ {\begin{array}{l}
 G\vert G^\ast \, \, \, H\mbox{\, \, is the root of \, \, } \tau_{I}^{\medstar}\, \, or\, \, G_2\, \, in\, \, \cfrac{\underline{G_2}}{G_1 }\left\langle {EC_\Omega ^\ast \, \, or\, \, ID_\Omega } \right\rangle
\in \overline \tau_{I}^{\medstar},  \\
 H_{j_k }^c \quad G_{b_k} \vert S_{j_k }^c \, \, in\, \, \tau _{{\rm {\bf
I}}_{\rm {\bf j}} }^\ast \in \overline \tau_{I}^{\medstar}\, \, for\, \, some\, \, 1\leqslant k\leqslant w, {\kern 1pt} \\
 \end{array}} \right.
\]
where,  $\overline \tau_{I}^{\medstar}{\kern 1pt}{\kern 1pt}{\kern
1pt}{\kern 1pt}{\kern 1pt}$ is the skeleton of $\tau_{I}^{\medstar}$,  which is defined by Definition 7.13 [6]. Then $\partial _{\tau_{I}^{\medstar} } \left( {G_{{\rm {\bf I}}_{\rm {\bf j}} }^\ast }
\right)\leqslant \partial _{\tau_{I}^{\medstar} } \left( {G_{b_k}
\vert S_{j_k }^c } \right)$ for some $1\leqslant k\leqslant w$ in $\tau
_{{\rm {\bf I}}_{\rm {\bf j}} }^\ast $;

(iii) Let $H\in \overline \tau_{I}^{\medstar} $ and $G\vert G^\ast
<\partial _{\tau_{I}^{\medstar} } \left( H \right)\leqslant H_I^V $
then $G_{H_I^V: H}^{\medstar(J)} \in \tau _I^{\medstar}$ and it is built up by
applying the separation algorithm along $H_I^V $ to $H$,  and is an upper
hypersequent of either $\left\langle {EC_\Omega ^\ast {\kern 1pt}}
\right\rangle $ if it is applicable,  or $\left\langle {ID_\Omega }
\right\rangle $ otherwise.

(iv) $S_j^c \in G_{I}^{\medstar}$ implies $H_j^c  \| H_i^c $ for all $H_i^c
\in I$ and,  $S_j^c \in G_{\mathcal{I}_{\rm {\bf j}} }^\ast $ for some $\tau
_{{\rm {\bf I}}_{\rm {\bf j}} }^\ast \in \tau_{I}^{\medstar}$.
\end{lemma}

\begin{proof} $\tau _I^{\medstar} $ is constructed by induction on
$o(I)$. For the base case,  let $o(I)=1$,  then $\tau _I^{\medstar} $ is built up
by Construction 7.3 and 7.7 in [6]. For the induction case,  suppose that
$o(I)\geqslant 2$,  $\tau _{I_l }^{\medstar} $ and $\tau _{I_r }^{\medstar} $ are
constructed such that Claims from (i) to (iv) hold.

Let $\cfrac{G'\vert S'\quad G''\vert S''}{G'\vert
G''\vert H'}(II)\in \tau ^\ast $,  where $G'\vert G''\vert H'\mbox{=}H_I^V $.
Then $I_l $ and $I_r $ occur in the left subtree $\tau ^\ast (G'\vert S')$
and right subtree $\tau ^\ast (G''\vert S'')$ of $\tau ^\ast (H_I^V )$,
respectively. Here,  almost all manipulations of the new main algorithm are
same as those of the old main algorithm. There are some caveats need to be
considered.

Firstly,  all leaves $\cfrac{\underline{\qquad\, \, }}{G\vert G^\ast }\left\langle {\tau ^\ast } \right\rangle \in \tau
_{I_l }^{\medstar}$ are replaced with $\tau _{{\rm {\bf I}}_{{\rm
{\bf j}}_r } }^\ast $  in Step 3 at Stage
1 in old main algorithm and,  $\cfrac{\underline{\qquad\, \, }}{G\vert G^\ast }\left\langle {\tau ^\ast } \right\rangle \in \tau _{I_r}^{\medstar}$ are replaced with $\tau _{I_l }^{\medstar} $ in Step 3 at
Stage 2.
Secondly,   we abandon the definitions of branch to $I$ and Notation 8.1 in [6] and then the symbol ${\rm {\bf I}}$ of the set  of branches,    which occur in $\tau _\mathbf{I}^{\medstar}$ in [6],  is replaced with $I$ in the new algorithm. We also replace $\Omega$ in $\tau _\mathbf{I}^{\Omega}$ with $\medstar$.
Thirdly,  under the new requirement that $I$ is complete,  we prove the
following property.

{\bf Property (A)}
$G_{I_l }^{\medstar}$ contains at most one copy of $G_{H_I^V: G''}^{\medstar(J)}\vert \widehat{S''}$.

\begin{proof}Suppose that there exist two copies $\left\{ {G_{H_I^V: G''}^{\medstar(J)} \vert
\widehat{S''}} \right\}_1 $ and $\left\{ {G_{H_I^V: G''}^{\medstar(J)} \vert
\widehat{S''}} \right\}_2 $ of $G_{H_I^V: G''}^{\medstar(J)} \vert \widehat{S''}$
in $G_{{\rm {\bf I}}_l }^{\medstar}$ and,  we put them into $\{\}_1 $ and $\{\}_2 $
in order to distinguish them. Let $\left[ S \right]_{G_{I_l }^{\medstar}
} $ be a splitting unit of $G_{{\rm {\bf I}}_l }^{\medstar}$ and $S$ its splitting
sequent. Then $\left| {v_l (S)} \right|+\left| {v_r (S)} \right|\geqslant
2$. Thus $S$ is a $(pEC)$-sequent and has the form $S_i^c $ by $\left[ S
\right]_{G_{I_l }^{\medstar}} \subseteq _c G\vert G^\ast $. Then
$[S]_{G_{I_l }^{\medstar}} =[S_i^c ]_{G_{I_l }^{\medstar}} $,  $H_i^c
\parallel H_j^c $ for all $H_j^c \in I_l $ and $S_i^c \in
G_{\mathcal{I}_{{\rm {\bf j}}_l } }^\ast $ for some $\tau _{{\bf I}_{{\rm {\bf
j}}_l } }^\ast \in \tau _{I_l }^{\medstar}$ by Claim (iv). Since $I_l $ is complete
and $G'\vert S'\leqslant H_{I_l }^V $,  then $H_i^c \parallel G'\vert S'$.

Let $\tau _{{\bf I}_{{\rm {\bf j}}_l } }^\ast $ be in the form
$\cfrac{\underline{G_{b_{l1} } \vert S_{j_{l1} }^c \, \, \, G_{b_{l2} } \vert S_{j_{l2} }^c \, \, \, \cdots\, \, \, G_{b_{lu} } \vert S_{j_{lu} }^c }}{G_{{\bf I}_{{\rm {\bf j}}_l }
}^\ast =\{{\kern 1pt}G_{b_{lk} } \}_{k=1}^u {\kern 1pt}\vert
G_{\mathcal{I}_{{\rm {\bf j}}_l } }^\ast }\left\langle {\tau _{{\bf I}_{{\rm {\bf
j}}_l } }^\ast } \right\rangle,
\cfrac{G_1 \vert S_1 \, \, \, G_2
\vert S_2 }{H_1 \equiv G_1 \vert G_2 \vert H''}(II)\in \tau ^\ast, $ where
$G_1 \vert S_1 \leqslant G'\vert S'$,  $G_2 \vert S_2 \leqslant H_i^c $,  $G_1
\vert G_2 \vert H''$ is the intersection node of $H_i^c $ and $G'\vert S'$,  as shown in Figure 3. Then
$\cfrac{\{G_{b_{lk} } \}_{k=1}^u {\kern 1pt}\vert \left\langle {G_1
\vert S_1 {\kern 1pt}} \right\rangle _{\mathcal{I}_{{\rm {\bf j}}_l } }
\, \, \, G_2 \vert S_2 }{H_2 \equiv
\{G_{b_{lk} } \}_{k=1}^u {\kern 1pt}\vert \left\langle {G_1 {\kern 1pt}}
\right\rangle _{\mathcal{I}_{{\rm {\bf j}}_l } } \vert G_2 \vert H''}(II)\in
\tau _{{\bf I}_{{\rm {\bf j}}_l } }^\ast $ by $G_1 \vert S_1 \leqslant G'\vert
S'\leqslant H_{I_l }^V $ and $S_i^c \in G_{\mathcal{I}_{{\rm {\bf j}}_l }
}^\ast $. Since $S_2 $ is separable in $G_{I_l }^{\medstar}$ by $G'\vert
S'\leqslant H_{I_l }^V $,  then $S_i^c \in G_2 \vert S_2 $ and $S_i^c $ is
not $S_2 $.

\[
\begin{array}{l}
 \qquad\quad\quad \, \,
 \ddots  \vdots  {\mathinner{\mkern0.5mu\raise0.7pt\hbox{.}\mkern0.0mu
 \raise2.1pt\hbox{.}\mkern0.8mu\raise3.6pt\hbox{.}}}
 \quad \quad\ddots  \vdots  {\mathinner{\mkern0.5mu\raise0.7pt\hbox{.}\mkern0.0mu
 \raise2.1pt\hbox{.}\mkern0.8mu\raise3.6pt\hbox{.}}}
 \quad\, \, \, \, \, \, \cdots\qquad \ddots  \vdots  {\mathinner{\mkern0.5mu\raise0.7pt\hbox{.}\mkern0.0mu
 \raise2.1pt\hbox{.}\mkern0.8mu\raise3.6pt\hbox{.}}} \\
 {\kern 1pt}{\kern 1pt}{\kern 1pt}{\kern 1pt}{\kern 1pt}{\kern 1pt}\tau
_{{\rm {\bf I}}_{{\rm {\bf j}}_l }}^\ast \left\{
{\begin{array}{l}
 {\kern 1pt}{\kern 1pt}G_{b_{l1} } \vert S_{j_{l1}}^c  \quad G_{b_{l2} } \vert
S_{j_{l2}}^c \quad \cdots
 \quad G_{b_{lu} } \vert
S_{j_{lu}}^c {\kern 1pt}{\kern 1pt}{\kern 1pt}{\kern 1pt}{\kern
1pt}{\kern 1pt}{\kern 1pt}{\kern 1pt}{\kern 1pt}{\kern 1pt}{\kern 1pt}{\kern
1pt} \\
 \quad\quad \qquad\qquad\ddots  \vdots  {\mathinner{\mkern0.5mu\raise0.7pt\hbox{.}\mkern0.0mu
 \raise2.1pt\hbox{.}\mkern0.8mu\raise3.6pt\hbox{.}}}
 \qquad\quad\qquad\qquad\qquad
 \ddots  \vdots  {\mathinner{\mkern0.5mu\raise0.7pt\hbox{.}\mkern0.0mu
 \raise2.1pt\hbox{.}\mkern0.8mu\raise3.6pt\hbox{.}}} \\
 \cfrac{ \qquad\qquad \{G_{b_{lk} }
\}_{k=1}^u {\kern 1pt}\vert \left\langle {G_1 \vert S_1 {\kern 1pt}}
\right\rangle _{\mathcal{I}_{{\rm {\bf j}}_l } }  \qquad\qquad
G_2 \vert S_2 }{H_2\equiv \{G_{b_{lk} }
\}_{k=1}^u {\kern 1pt}\vert \left\langle {G_1 {\kern 1pt}} \right\rangle
_{\mathcal{I}_{{\rm {\bf j}}_l }} \vert G_2 \vert
H''}(II) \\
 \qquad\qquad  \qquad\qquad \ddots  \vdots  {\mathinner{\mkern0.5mu\raise0.7pt\hbox{.}\mkern0.0mu
 \raise2.1pt\hbox{.}\mkern0.8mu\raise3.6pt\hbox{.}}} \\
\qquad\qquad G_{{\rm {\bf I}}_{{\rm {\bf j}}_l }
}^\ast =\{{\kern 1pt}G_{b_{lk} } \}_{k=1}^u {\kern 1pt}\vert
G_{\mathcal{I}_{{\rm {\bf j}}_l }}^\ast {\kern 1pt} \\
 \end{array}} \right. \\
  \qquad\quad  \qquad\qquad\qquad\qquad
\ddots  \vdots  {\mathinner{\mkern0.5mu\raise0.7pt\hbox{.}\mkern0.0mu
 \raise2.1pt\hbox{.}\mkern0.8mu\raise3.6pt\hbox{.}}} \\
 {\kern 1pt}{\kern 1pt}{\kern 1pt}{\kern 1pt}{\kern 1pt}{\kern 1pt} \quad\qquad \qquad\qquad\qquad\quad\, \, \,  G_{I_l }^{\medstar}\\
 \end{array}
\]
\begin{center}
Figure 3 A fragment  of $\tau _{I_l}^{\medstar} $
\end{center}

 {\bf Property  (B)}
The set of splitting sequents of $\left[ {S_i^c } \right]_{G_{I_l }^{\medstar}} $ is equal to that of $\left[ {S_i^c } \right]_{G_2 \vert S_2 }
$.
\begin{proof}
Let $\cfrac{G_1 '\vert S_1 '\, \, \, \, \, G_2 '\vert S_2 '}{H_1 '\equiv G_1 '\vert G_2 '\vert H'''}(II)\in \tau
^\ast $,  $G_1 '\vert S_1 '{\kern 1pt}{\kern 1pt}\leqslant H_1 $ and $S_1
'\in \left\langle {G_1 '\vert S_1 '} \right\rangle
_{\mathcal{I}_{{\rm {\bf j}}_l } } $. Then $S_1 '$ and $S_2 '$ are separable
in $G_{I_l }^{\medstar}$. Thus $G_{H_1 ': G_2 '}^{\medstar(J)} \vert
\widehat{S_2 '}\subseteq G_{I_l }^{\medstar}$ is closed. Hence $G_{H_1
: G_2 }^{\medstar(J)} \vert \widehat{S_2 }-\bigcup _{G_2 '\vert S_2 '} G_{H_1 ': G_2
'}^{\medstar(J)} \vert \widehat{S_2 '}$ is closed,  where $G_2 '\vert S_2 '$ in
$\bigcup _{G_2 '\vert S_2 '} $ runs over all $II\in \tau ^\ast $ above such
that $G_{H_1 ': G_2 '}^{\medstar(J)} \vert \widehat{S_2 '}\subseteq G_{H_1: G_2
}^{\medstar(J)} \vert \widehat{S_2 }$. Therefore $v(G_{H_1: G_2 }^{\medstar(J)} \vert\widehat{S_2 }-\bigcup _{G_2 '\vert S_2 '} G_{H_1 ': G_2 '}^{\medstar(J)} \vert
\widehat{S_2 '})=v(G_2 \vert S_2 )$,  $\{S_j^c: S_j^c \in G_2 \vert S_2
, H_j^c \geqslant G_2 \vert S_2 \}=\{S_j^c: S_j^c \in G_{H_1: G_2 }^{\medstar(J)}
\vert \widehat{S_2 }-\bigcup _{G_2 '\vert S_2 '} G_{H_1 ': G_2 '}^{\medstar(J)} \vert\widehat{S_2 '}\}$ and $\left[ {S_i^c } \right]_{G_{I_l }^{\medstar}}
\subseteq G_{H_1: G_2 }^{\medstar(J)} \vert \widehat{S_2 }-\bigcup _{G_2 '\vert S_2 '}G_{H_1 ': G_2 '}^{\medstar(J)} \vert \widehat{S_2 '}$. Then the set of splitting
sequents of $\left[ {S_i^c } \right]_{G_{I_l }^{\medstar}} $ is equal to
that of $\left[ {S_i^c } \right]_{G_2 \vert S_2 } $ since each splitting
sequent $S'''\in \left[ {S_i^c } \right]_{G_{I_l }^{\medstar}} $ is a
$(pEC)$-sequent by $\left| {v_l (S''')} \right|+\left| {v_r (S''')}
\right|\geqslant 2$ and $S'''\in _c G\vert G^\ast $.  This completes the proof of  Property (B).
\end{proof}

We therefore assume that,  without loss of generality,  $S_i^c $ is in the
form $\Gamma, p_k, \Delta \Rightarrow p_k $ by Property (B),  Lemma 3.16 and
the observation that each derivation-splicing operation is local. There are
two cases to be considered in the following.

\textbf{Case 1} $S_1 \notin \left\langle {G_1 \vert S_1 } \right\rangle
_{G_b \vert S_j^c } $ for all $\tau _{G_b \vert S_j^c }^\ast \in \tau
_{H_I^V: G''}^{\medstar(J)} $,  $G_1 \vert S_1 \leqslant H_j^c \leqslant H_I^V $.
Then $G_{H_1: G_2 }^{\medstar(J)} \bigcap G_{H_I^V: G''}^{\medstar(J)} =\emptyset $. We
assume that,  without loss of generality,  $\left\langle {G_2 \vert S_2 }
\right\rangle _k^- =G_2 '\vert \Gamma \Rightarrow t$,  $\left\langle {G_2
\vert S_2 } \right\rangle _k^+ =G_2 ''\vert S_2 \vert \Delta \Rightarrow
t$.  Then $\left\langle {G_{I_l }^{\medstar}} \right\rangle _k^- =G_{H_2: G_2'}^{\medstar(J)} \vert \Gamma \Rightarrow t$ since $S=\Gamma, p_k, \Delta
\Rightarrow p_k $ isn't a focus sequent at all nodes from $G_2 \vert S_2 $
to $G_{I_l }^{\medstar}$ in $\tau _{I_l}^{\medstar}$ and,  $H_j^c \leqslant H_1 $ or $H_j^c \vert \vert G_1 \vert S_1 $ for all $S_j^c \in G_2
'$ by Lemma 6.7 in [6]. Thus $\left\langle {G_{I_l }^{\medstar}}
\right\rangle _k^- \backslash \Gamma \Rightarrow t\subseteq G_{H_2: G_2
}^{\medstar(J)} $. Therefore $\left\{ {G_{H_I^V: G''}^{\medstar(J)} \vert \widehat{S''}}
\right\}_1 \vert \left\{ {G_{H_I^V: G''}^{\medstar(J)} \vert \widehat{S''}}
\right\}_2 \subseteq \left\langle {G_{I_l }^{\medstar}} \right\rangle
_k^+ $ because $\left[ S \right]_{G_{I_l }^{\medstar}} \subseteq G_{H_2
: G_2 }^{\medstar(J)} \vert \widehat{S_2 }$,  $G_{H_2: G_2 }^{\medstar(J)} \vert
\widehat{S_2 }\bigcap( \left\{ {G_{H_I^V: G''}^{\medstar(J)} \vert \widehat{S''}}
\right\}_1 \vert \left\{ {G_{H_I^V: G''}^{\medstar(J)} \vert \widehat{S''}}
\right\}_2 )=\emptyset $ and $\left\langle {G_{I_l }^{\medstar}}
\right\rangle _k^- \backslash \{\Gamma \Rightarrow t\}\vert \left\langle
{G_{I_l }^{\medstar}} \right\rangle _k^+ \backslash \{\Delta \Rightarrow
t\}\vert \Gamma, p_k, \Delta \Rightarrow p_k =G_{I_l }^{\medstar}$. This
shows that any splitting unit $\left[ S \right]_{G_{I_l }^{\medstar}} $
outside $G_{H_I^V: G''}^{\medstar(J)} \vert \widehat{S''}$ in $G_{{\rm {\bf I}}_l
}^{\medstar}$ doesn't take two copies of $G_{H_I^V: G''}^{\medstar(J)} \vert \widehat{S''}$
apart,  i.e.,  the case of $\left\{ {G_{H_I^V: G''}^{\medstar(J)} \vert
\widehat{S''}} \right\}_1 \subseteq \left\langle {G_{I_l }^{\medstar}}
\right\rangle _k^- $ and $\left\{ {G_{H_I^V: G''}^{\medstar(J)} \vert
\widehat{S''}} \right\}_2 \subseteq \left\langle {G_{I_l }^{\medstar}}
\right\rangle _k^+ $ doesn't happen.

\textbf{Case 2} $S_1 \in \left\langle {G_1 \vert S_1 } \right\rangle _{G_b
\vert S_j^c } $ for some $\tau _{G_b \vert S_j^c }^\ast \in \tau _{H_I^V
: G''}^{\medstar(J)} $,  $G_1 \vert S_1 \leqslant H_j^c \leqslant H_I^V $. Then\\
 $G_b
\vert \left\langle {G_1 } \right\rangle _{S_j^c } \vert G_2 \vert H''\in
\tau _{G_b \vert S_j^c }^\ast $. Thus $G_{H_1: G_2 }^{\medstar(J)} \vert
\widehat{S_2 }\subseteq G_{H_I^V: G''}^{\medstar(J)} \vert \widehat{S''}$. Hence
$\left[ {S_i^c } \right]_{G_{I_l }^{\medstar}} \subseteq G_{H_I^V
: G''}^{\medstar(J)} \vert \widehat{S''}$. The case of $S_i^c \in G''$ is tackled
with the same procedure as the following. Let $\left[ {S_i^c }
\right]_{G_{I_l }^{\medstar}} \subseteq \left\{ {G_{H_I^V: G''}^{\medstar(J)}
\vert \widehat{S''}} \right\}_1 $. Then there exists a copy of $\left[ S
\right]_{G_{I_l }^{\medstar}} $ in $\left\{ {G_{H_I^V: G''}^{\medstar(J)} \vert
\widehat{S''}} \right\}_2 $ and let $\Gamma, p_{k'}, \Delta \Rightarrow
p_{k'} $ be its splitting sequent. We put two splitting units into $\{\}_k $
and $\{\}_{k'} $ in order to distinguish them. Then $\{\left[ S
\right]_{G_{I_l }^{\medstar}} \}_k \subseteq \left\{ {G_{H_I^V
: G''}^{\medstar(J)} \vert \widehat{S''}} \right\}_1 $ and $\{\left[ S
\right]_{G_{I_l }^{\medstar}} \}_{k'} \subseteq \left\{ {G_{H_I^V
: G''}^{\medstar(J)} \vert \widehat{S''}} \right\}_2 $. We assume that,  without loss
of generality,  $\left\langle {G_2 \vert S_2 } \right\rangle _k^- =G_2 '\vert
\Gamma \Rightarrow t$,  $\left\langle {G_2 \vert S_2 } \right\rangle _k^+
=G_2 ''\vert S_2 \vert \Delta \Rightarrow t$. Then $\left\langle {G_{I_l }^{\medstar}} \right\rangle _k^- \backslash \{\Gamma \Rightarrow
t\}\subseteq \left\{ {G_{H_I^V: G''}^{\medstar(J)} \vert \widehat{S''}} \right\}_1
$.  Thus $\{\left[ S \right]_{G_{I_l }^{\medstar}} \}_{k'} \subseteq
\left\{ {G_{H_I^V: G''}^{\medstar(J)} \vert \widehat{S''}} \right\}_2 \subseteq
\left\langle {G_{I_l }^{\medstar}} \right\rangle _k^+ $ by $\left\langle
{G_{I_l }^{\medstar}} \right\rangle _k^- \backslash \{\Gamma \Rightarrow
t\}\bigcup \left\langle {G_{I_l }^{\medstar}} \right\rangle _k^+ \backslash
\{\Delta \Rightarrow t\}=G_{I_l }^{\medstar}\backslash \Gamma, p_k
, \Delta \Rightarrow p_k $. Then $\left\langle {\left\langle {G_{I_l }^{\medstar}} \right\rangle _k^+ } \right\rangle _{k'}^- =\left\langle
{G_{I_l }^{\medstar}} \right\rangle _{k'}^- $,  $\{\Delta \Rightarrow
t\}_k \vert \{\Delta \Rightarrow t\}_{k'} \subseteq \left\langle
{\left\langle {G_{I_l }^{\medstar}} \right\rangle _k^+ } \right\rangle
_{k'}^+ $ where,  we put two copies of $\Delta \Rightarrow t$ into $\{\}_k $
and $\{\}_{k'} $ in order to distinguish them. Then $\Gamma \Rightarrow t\in
\left\langle {G_{I_l }^{\medstar}} \right\rangle _{k'}^- $,  $\vdash
_{{\rm {\bf GL}}} \left\langle {G_{I_l }^{\medstar}} \right\rangle _k^-
$,  $\vdash _{{\rm {\bf GL}}} \left\langle {G_{I_l }^{\medstar}}
\right\rangle _{k'}^- $ and $\left\langle {G_{I_l }^{\medstar}}
\right\rangle _{k'}^- $ is a copy of $\left\langle {G_{I_l }^{\medstar}}
\right\rangle _k^- $. Then $\mathcal{D}(\left\langle {G_{I_l }^{\medstar}
} \right\rangle _k^- )=\mathcal{D}(\left\langle {G_{I_l }^{\medstar}}
\right\rangle _{k'}^- )\subseteq \mathcal{D}(G_{I_l }^{\medstar})$ could
be cut off one of them because they are two same sets of hypersequents in
$\mathcal{D}(G_{I_l }^{\medstar})$. Meanwhile,  two copies of $\Delta
\Rightarrow t$ in $\left\langle {\left\langle {G_{I_l }^{\medstar}}
\right\rangle _k^+ } \right\rangle _{k'}^+ $ can't be taken apart by any
splitting unit outside $G_{H_I^V: G''}^{\medstar(J)} \vert \widehat{S''}$ in
$G_{I_l }^{\medstar}$ by the reason as shown in Case 1 and thus could be
contracted into one by $(EC)$ in $\mathcal{D}(G_{I_l }^{\medstar})$.
Therefore two copies $\left\{ {G_{H_I^V: G''}^{\medstar(J)} \vert \widehat{S''}}
\right\}_1 $ and $\left\{ {G_{H_I^V: G''}^{\medstar(J)} \vert \widehat{S''}}
\right\}_2 $ of $G_{H_I^V: G''}^{\medstar(J)} \vert \widehat{S''}$ can be
contracted into one in $G_{I_l }^{\medstar}$ by $\left\langle {EC_\Omega
^\ast } \right\rangle $. This completes the proof of  Property (A).
\end{proof}
With Property (A),    all manipulations in the old main algorithm in [6]
work well.  This completes the construction of $\tau_{I}^{\medstar} $
and the proof of Theorem 4.1.
\end{proof}

\begin{theorem}
The standard completeness  holds for ${\rm {\bf HpsUL}}^\ast $.
\end{theorem}
\begin{proof}
Let  $\underset{}{\overset{i}{\longleftrightarrow}}$ denote the
$i$-th logical link of  iff  in the following. $\models_\mathcal{K} A$ means that
$v(A)\geqslant t$ for every algebra $\mathcal{A}$ in $\mathcal{K}$ and
valuation $v$ on $\mathcal{A}$.  Let ${\rm {\bf psUL}}^\ast $,  $\mbox{LIN(}{\rm {\bf psUL}}^\ast )$,  ${\rm {\bf psUL}}^{\ast D}$ and $[0, 1]_{{\rm {\bf psUL}}^\ast } $ denote the
classes of all ${\rm {\bf psUL}}^\ast $-algebras,  ${\rm {\bf psUL}}^\ast
$-chain,  dense ${\rm {\bf psUL}}^\ast $-chain and standard ${\rm {\bf
psUL}}^\ast $-algebras ( i.e.,  their lattice reducts are $[0, 1]$),  respectively. We have an inference sequence,  as shown in
Figure 4.
\[
\boxed {\begin{array}{l}
 \vdash _{{\rm {\bf HpsUL}}^\ast } A\underset{}{\overset{1^\circ
}{\longleftrightarrow}}\vdash _{{\rm {\bf GpsUL}}^\ast } \Rightarrow
A\underset{}{\overset{2^\circ }{\longleftrightarrow}}\vdash _{{\rm {\bf
GpsUL}}^{\ast D}} \Rightarrow A\underset{}{\overset{3^\circ
}{\longleftrightarrow}}\models _{{\rm {\bf psUL}}^{\ast D}} A \\
\updownarrow 1\qquad\qquad\qquad\qquad\qquad\qquad\qquad
\qquad\qquad\qquad\updownarrow 4^{\circ} \\
 \models_{{\rm {\bf psUL}}^\ast }
A\underset{}{\overset{2}{\longleftrightarrow}}\models_{\mathrm{LIN}(\rm {\bf
psUL^{\ast}} )} A\underset{}{\overset{3}{\longleftrightarrow}}\models_{\rm {\bf
psUL^{\ast D}}} A \underset{}{\overset{4}{\longleftrightarrow}}\models_{[0, 1]_{{\rm {\bf psUL}}^\ast } } A \\
 \end{array}}
\]
\begin{center}
Figure 4  Two ways to prove standard completeness
\end{center}
Links  from 1 to 4 show Jenei and Montagna's  algebraic method to prove standard completeness  and currently,   it seems hopeless to built up the link 3,  see [7$\sim$10].  Links  from $1^{\circ}$ to $4^{\circ}$ show Metcalfe and Montagna's proof-theoretical method.
Density elimination is at Link $2^\circ $ in Figure 4 and other links are
proved by standard procedures with minor revisions and omitted,  see [1,  4].
\end{proof}  

\section*{References}
\bibliographystyle{elsarticle-harv}

\end{document}